\documentclass[11pt]{amsart}

\usepackage{amsfonts}
\usepackage{amsmath}
\usepackage{amssymb}
\usepackage{amscd}
\usepackage{color}
\usepackage{xcolor}
\usepackage{amsthm}
\usepackage[hmargin=4cm,vmargin=4cm]{geometry}

\usepackage{enumitem}

\usepackage[hyperindex]{hyperref}

\usepackage{etoolbox}
\makeatletter
\let\ams@starttoc\@starttoc
\makeatother
\usepackage{parskip}
\makeatletter
\let\@starttoc\ams@starttoc
\patchcmd{\@starttoc}{\makeatletter}{\makeatletter\parskip\z@}{}{}
\makeatother

\newtheorem{thm}{Theorem}

\newtheorem{thm*}{Theorem}
\newtheorem{prop}{Proposition}
\newtheorem{lma}[prop]{Lemma}

\newtheorem{cor}[prop]{Corollary}

\theoremstyle{definition}

\newtheorem{df}[prop]{Definition} 

\theoremstyle{remark}

\newtheorem{rmk}[prop]{Remark} 

\newcommand{\F}{{\mathbb{F}}}
\newcommand{\R}{{\mathbb{R}}}
\newcommand{\Z}{{\mathbb{Z}}}
\newcommand{\C}{{\mathbb{C}}}
\newcommand{\Q}{{\mathbb{Q}}}
\newcommand{\D}{{\mathbb{D}}}
\newcommand{\bK}{{\mathbb{K}}}
\newcommand{\bF}{{\mathbb{F}}}

\newcommand{\del}{\partial}

\newcommand{\sm}[1]{C^\infty(#1)}

\newcommand{\K}{\mathcal{K}}

\newcommand{\cL}{\mathcal{L}}

\newcommand{\til}[1]{\widetilde{#1}}
\newcommand{\wh}[1]{\widehat{#1}}
\renewcommand{\check}[1]{\widehat{#1}}
\newcommand{\ol}[1]{\overline{#1}}

\newcommand{\zp}{{\mathbb{Z}/(p)}}

\newcommand{\rp}{{\mathcal{R}_p}}
\newcommand{\hrp}{{\widehat{\cl R}_p}}
\newcommand{\lamzero}{\Lambda^0}
\newcommand{\lambzero}{\Lambda^0}

\newcommand{\tphi}{{\til{\phi}}}

\DeclareMathOperator{\ima}{\mathrm{im}}

\newcommand{\tmin}{{\text{min},\bK}}
\newcommand{\tmon}{{\text{mon},\bK}}
\newcommand{\tuniv}{{\text{univ},\bK}}

\newcommand{\om}{\omega}

\newcommand{\eps}{\epsilon}
\newcommand{\De}{\Delta}

\newcommand{\extheta}{\left < \theta \right >}

\newcommand{\cA}{\mathcal{A}}
\newcommand{\cB}{\mathcal{B}}
\newcommand{\cC}{\mathcal{C}}
\newcommand{\cD}{\mathcal{D}}
\newcommand{\cE}{\mathcal{E}}
\newcommand{\cF}{\mathcal{F}}

\newcommand{\cH}{\mathcal{H}}

\newcommand{\cJ}{\mathcal{J}}
\newcommand{\cK}{\mathcal{K}}
\newcommand{\cO}{\mathcal{O}}

\newcommand{\cP}{\mathcal{P}}
\newcommand{\cR}{\mathcal{R}}

\newcommand{\fix}{\mathrm{Fix}}

\def\mrm#1{{\mathrm{#1}}}
\def\bb#1{{\mathbb{#1}}}
\def\cl#1{{\mathcal{#1}}}

\newcommand{\bra}[1]{{\{ #1 \}}}
\newcommand{\brat}[1]{{\left< #1 \right>}}
\newcommand{\tens}{\otimes}

\DeclareMathOperator{\supp}{\mathrm{supp}}

\DeclareMathOperator{\Ham}{\mathrm{Ham}}

\DeclareMathOperator{\Tor}{\mathrm{Tor}^1}
\DeclareMathOperator{\Symp}{\mathrm{Symp}}

\DeclareMathOperator{\im}{\mathrm{im}}

\DeclareMathOperator{\id}{\mathrm{id}}
\DeclareMathOperator{\spec}{\mathrm{Spec}}
\DeclareMathOperator{\Spec}{\mathrm{Spec}}
\DeclareMathOperator{\loc}{\mathrm{loc}}
\DeclareMathOperator{\tot}{\mathrm{tot}}

\DeclareMathOperator{\pemod}{{\mathbf{pmod}}}
\DeclareMathOperator{\barc}{{\mathbf{barcodes}}}

\DeclareMathOperator{\ind}{\mathrm{ind}}

\def\Hk{H^{(p)}}

\def\H2{H^{(2)}}

\newcommand{\esemail}{shelukhin@dms.umontreal.ca}

\begin{document}

	\title{On the Hofer-Zehnder conjecture}

	\author{Egor Shelukhin}
	\address{Egor Shelukhin, Department of Mathematics and Statistics,
		University of Montreal, C.P. 6128 Succ.  Centre-Ville Montreal, QC
		H3C 3J7, Canada}
	\email{\esemail}

\begin{abstract}
We prove that if a Hamiltonian diffeomorphism of a closed monotone symplectic manifold with semisimple quantum homology has more contractible fixed points, counted homologically, than the total dimension of the homology of the manifold, then it must have an infinite number of contractible periodic points. This constitutes a higher-dimensional homological generalization of a celebrated result of Franks from 1992, as conjectured by Hofer and Zehnder in 1994.			
	
\end{abstract}


	
\maketitle

\tableofcontents

\section{Setup}

\begin{df}
We call a closed symplectic manifold $(M,\om)$ {\em monotone} if the class of the symplectic form and the first Chern class are positively proportional \[[\om] = \kappa \cdot c_1(TM),\] for $\kappa > 0,$ on the image of the Hurewicz map $\pi_2(M) \to H_2(M,\Z).$ By suitably rescaling $\om,$ we can always assume that $\kappa = 2,$ which will be our convention in this paper.
\end{df}

\begin{rmk}
We mostly do not consider closed symplectic manifolds with $[\om] = 0,$ or $c_1(TM) = 0$ or when the constant $\kappa$ is negative, since in this case all Hamiltonian diffeomorphisms have an infinite number of contractible periodic points by the work of Hingston \cite{Hingston-torus}, Ginzburg \cite{Ginzburg-CC}, Ginzburg-G\"{u}rel and others (\cite{GG-revisited} and references therein), on the Conley conjecture, whence our main results are tautologically true. Indeed, our results are pertinent to the case when the Conley conjecture does {\em not} hold. 
\end{rmk}

For a diffeomorphism $\phi \in \Ham(M,\om),$ we denote $\fix(\phi) = \{ x \in M\,|\, \phi(x) =x \}.$ Given a Hamiltonian isotopy $\{\phi^t_H\},$ $\phi^1_H = \phi,$ for $H \in \cH = \sm{[0,1] \times M, \R},$ which we usually take to be normalized so that $H(t,-)$ has zero mean for all $t \in [0,1],$ there is a bijective correspondence between $\fix(\phi)$ and the $1$-periodic orbits of the isotopy. We shall use the two notions interchangeably. It is a consequence of Floer theory that the homotopy class of the loop $\alpha(x,\phi) = \{\phi_H^t(x)\}$ for $x\in \fix(\phi)$ does not depend on the choice of the Hamiltonian $H$ with $\phi^1_H = \phi.$ We shall consider only orbits $x$ that are {\em contractible}, that is $\alpha(x,\phi)$ is in the free homotopy class of a point. In fact, for brevity, we keep the notation $\fix(\phi)$ for the set of contractible fixed points of $\phi.$ Finally, we say that $\phi$ has isolated fixed points if $\fix(\phi)$ is a finite set.

We call $x$ a $k$-periodic point of $\phi$ if $\phi^k(x) = x,$ and we call it {\em simple} if $\phi^m(x) \neq x$ for all proper divisors $m$ of $k.$ Note that if $m$ divides $k$ then $\fix(\phi^m) \subset \fix(\phi^k).$ We denote by $x^{(k)} \in \fix(\phi^k)$ the image of $x \in \fix(\phi)$ under the inclusion $\fix(\phi) \subset \fix(\phi^k).$ Finally $x$ is a periodic point if it is $k$-periodic for some $k.$

To an isolated fixed point $x \in \fix(\phi)$ for $\phi \in \Ham(M,\om)$ one can associate a homology group, called the local Floer homology $HF^{\loc}(\phi,x)$ of $\phi$ at $x.$ While the definition shall be discussed in detail in Section \ref{subsubsec: local FH}, for now we note that if $x \in \fix(\phi)$ is a {\em non-degenerate} fixed point, that is $\ker(D\phi_x - \id) = 0,$ then as an ungraded $\bK$-module,  \begin{equation}\label{eq: HF loc non-degenerate}HF^{\loc}(\phi,x) \cong \bK.\end{equation} This suggests that $\dim_{\bK} HF^{\loc}(\phi,x)$ is a reasonable homological count of the fixed point $x.$

Furthermore, to $\phi \in \Ham(M,\om)$ with isolated fixed points, one can associate a filtered Floer homology theory \cite{Floer1,Floer2,Floer3} with coefficient field $\bK,$ with nice finiteness properties. Depending on the degree of detail of the data that we wish to extract, it may depend on a lift $\til{\phi}$ of $\phi$ to the universal cover $\til{\Ham}(M,\om).$ An essentially complete description of the resulting system of vector spaces and maps between them, formalized as a {\em persistence module}, consists of a finite multiset $\cB(\til{\phi},\bK) = \{(I_j, m_j)\},$ $m_j \geq 1,$ of intervals $I_j$ in $\R$ that are either finite, or infinite on the right. We call this multiset the Floer {\em barcode} associated to $\til{\phi}.$ We describe these notions, initially introduced in symplectic topology by Polterovich and Shelukhin \cite{PolShe} (see also \cite{PolSheSto, UsherZhang, FOOO-polydiscs}), in detail in Section \ref{sec: prelim}. For the time-being we denote by $K(\phi,\bK)$ the number of finite bars in $\cB(\til{\phi},\bK),$ and call their lengths $\beta_1(\phi,\bK) \leq \ldots \leq \beta_{K(\phi,\bK)}(\phi,\bK)$ arranged in increasing order the {\em bar-length spectrum} of $\phi.$ It depends only on $\phi.$ In particular so does $\beta(\phi,\bK) = \beta_{K(\phi,\bK)}(\phi,\bK),$ called Usher's {\em boundary depth} \cite{UsherBD1,UsherBD2} of $\phi,$ as well as the sum $\beta_{\tot}(\phi,\bK) = \beta_1(\phi,\bK) + \ldots + \beta_{K(\phi,\bK)}(\phi,\bK)$ of all the bar-lengths. Furthermore, a closely related notion introduced by Viterbo, Schwarz, and Oh \cite{Viterbo-specGF,Schwarz:action-spectrum,Oh-specnorm} is the spectral norm $\gamma(\phi,\bK)$ of $\phi \in \Ham(M,\om),$ which in particular satisfies the property $\gamma(\phi\phi') \leq \gamma(\phi) + \gamma(\phi')$ for all $\phi,\phi' \in \Ham(M,\om),$ and, by \cite{KS-bounds}, $\gamma(\phi) \geq \beta(\phi).$ These notions play an important role in this paper. Finally, we remark that a priori, the barcode of $\til{\phi}$ depends on a choice of ground coefficient field $\bK$ for Floer theory. However, it is easy to see that it remains invariant under field extensions. Furthermore, when the choice of field is immaterial for the statement, or is clear from the context, we omit it from the notation.

\section{Outline}

Our main result is as follows. It provides a solution to a conjecture of Hofer and Zehnder \cite[p. 263]{HZ-book}, the number of fixed points being interpreted homologically, for the broad class of symplectic manifolds whose even quantum homology algebra is semisimple. This class contains among other symplectic manifolds the complex projective spaces, the complex Grassmannians, and is closed under monotone products (cf. \cite{EntovPolterovichCalabiQM}). Furthermore, our result is effective, allowing to estimate from below the growth rate of the number of periodic points.

\begin{thm}\label{thm: main2}
	Let $(M,\omega)$ be a closed monotone symplectic manifold whose even quantum homology algebra over a ground field $\bK$ is semisimple. Then a Hamiltonian diffeomorphism $\phi \in \Ham(M,\om),$ whose set $\fix(\phi)$ of contractible fixed points is finite, and \[\sum_{x\in \fix(\phi)} \dim_{\bK} HF^{\loc}(\phi,x) > \dim_{\bK} H_*(M,\bK),\] must have an infinite number of contractible periodic points. In fact, if $\bK$ is of characteristic zero, then $\phi$ has a simple contractible $p$-periodic point for each sufficiently large prime $p.$
	\end{thm}

From now on, for $\phi \in \Ham(M,\om)$ with isolated fixed points, and a ground field $\bK,$ we denote \begin{equation}\label{eq: homological count N} N(\phi,\bK) = \sum_{x\in \fix(\phi)} \dim_{\bK} HF^{\loc}(\phi,x).\end{equation}

\begin{rmk}
By \eqref{eq: HF loc non-degenerate}, if each fixed point $x \in \fix(\phi)$ is non-degenerate, then the homological count $N(\phi,\bK)$ of the contractible fixed points of $\phi$ simplifies to the set-theoretic count $\# \fix(\phi).$ Furthermore, the solution of the homological Arnol'd conjecture in the monotone case \cite{Floer3} implies by a simple algebraic exercise that if $\phi$ has isolated fixed points, then $N(\phi,\bK) \geq \dim_{\bK} H_*(M,\bK)$ for all coefficient fields $\bK.$ This suggests that $N(\phi,\bK)$ is a reasonable homological count of the fixed points of $\phi,$ and in turn motivates Theorem \ref{thm: main2} as a solution, for a class of symplectic manifolds, to the homological Hofer-Zehnder conjecture.
\end{rmk}

\begin{rmk}
Recall that the result of Franks \cite{Franks-sphere, Franks-NY} implies, in the smooth case, that $\phi \in \Ham(S^2)$ with $\# \fix(\phi) > 2$ has an infinite number of periodic points. This result prompted the conjecture of Hofer and Zehnder from 1994, which has since remained completely open. Strong lower bounds on the growth rate in the setting of the Franks theorem were provided in \cite{FranksHandel-growthrate, LeCalvez-growthrate1, LeCalvez-growthrate2}. Since $\dim H_*(S^2) = 2,$ Theorem \ref{thm: main2} implies the theorem of Franks for example in the non-degenerate case. In fact, the only case where the smooth Franks theorem does not follow from our result directly is the case when $\# \fix(\phi) > 2,$ but $\sum \dim_{\bK} HF^{\loc}(\phi,x) = 2.$ This implies having fixed points $x$ that are not {\em homologically visible}, that is $\dim HF^{\loc}(\phi,x) = 0.$ It would be very interesting to understand how to treat this case: while for $S^2$ there exist symplectic proofs of the complete smooth Franks theorem \cite{BramhamHofer-FirstSteps, Kerman-etal}, in higher dimensions this currently seems to be out of reach. Finally, interpreting the Hofer-Zehnder conjecture more generally, Ginzburg, G\"{u}rel \cite{GG-nc, Gurel-nc} and Orita \cite{Orita1,Orita2}, have shown that in certain settings the existence of a {\em non-contractible} periodic orbit implies the existence of an infinite number of periodic points, and, in a different direction, so does the existence of a fixed point that is isolated as an invariant set \cite{GG-hyperbolic,GG-pseudorotations}. We expect our results and methods to generalize to settings of this kind, and to yield new cases of the Conley conjecture. 
\end{rmk}

\begin{rmk}
Our result deduces the main conclusion by the method of the filtered $\zp$-equivariant product-isomorphism, from a bound \begin{equation}\label{eq:gamma bound}\beta(\phi^k) \leq \mrm{const}\end{equation} on the boundary depth \cite{UsherBD1, UsherBD2} of the Floer barcode of $\phi^k$ for certain iterations $k \geq k_0,$ for $k_0$ sufficiently large. In fact, assuming semisimplicity of the quantum homology, we obtain such a uniform bound for {\em all} Hamiltonian diffeomorphisms $\psi \in \Ham(M,\om)$ and hence in particular for $\psi = \phi^k$ for all $k.$ This allows us to estimate the growth rate of periodic points explicitly in the characteristic zero case. Finally, it is not hard to observe that our methods only require the weaker condition of strictly sublinear growth $\lim_{m \to \infty} \frac{1}{k_m} \beta(\phi^{k_m}) = 0$ on suitable subsequences.
\end{rmk}

\begin{rmk}\label{rmk: thm main2}
A few technical remarks on Theorem \ref{thm: main2}.

\begin{enumerate}[label = \roman*.]
\item

The even rational quantum homology of $(M,\omega)$ is semisimple if and only if its degree $2n$ component is semisimple \cite[Theorem 5.1]{EntovPolterovich-semisimple}.

\item
 By the prime number theorem there are at least $\sim \frac{k^2}{\log(k)}$ periodic points of period $\leq k.$ Note that having one simple $p$-periodic point automatically implies having a $p$-tuple, hence the sum  \[ \sum_{k_0 \leq p \leq k} p \] over primes $p$ provides a lower bound on the number of periodic points of $\phi$ up to period $k,$ by Theorem \ref{thm: main2}.

\end{enumerate}

\end{rmk}

We record for the reader's convenience the precise statement of the uniform bound on boundary depth that goes into the proof of Theorem \ref{thm: main2}. The statement applies to an arbitrary coefficient field $\bK$, however we shall mainly use it for $\bK = \F_p$ for $p$ prime.



\begin{thm}\label{thm: bound semisimple}
Let $\bK$ be a field. Suppose that the even quantum homology $QH_{ev}(M,\bK)$ of $(M,\om)$ with ground field $\bK$ is semisimple. Then the boundary depth of each $\psi \in \Ham(M,\om)$ satisfies \[ \beta(\psi) \leq 8n.\] 
\end{thm}

\begin{rmk}
It is interesting to note that the example of $(M,\om) = (S^2 \times S^2, \om_{st} \oplus \om_{st})$ shows \cite[Theorem 6.2.6]{PolterovichRosen} that, while Theorem \ref{thm: bound semisimple} works for $\beta(\psi),$ it definitely does not work for $\gamma(\psi),$ as there exist $\psi \in \Ham(M,\om)$ with $\gamma(\psi^k) \sim k,$ as $k \to \infty.$  
\end{rmk}

We include an outline of the proof of the main result in the key case of $\bK = \Q,$ to give a general sense of its strategy. The proof of the general case, in technical detail, appears in Section \ref{sec: proof}.

We start from the condition that $N(\phi,\Q) > \dim_{\Q} H_*(M,\Q).$ This implies by the theory of Floer barcodes that the number of endpoints of finite bars in the barcode of $\phi$ is positive, and furthermore that the boundary depth $\beta(\phi,\Q)$ of $\phi$ is positive. By an argument of reduction to positive characteristic, we consequently obtain that for all primes $p$ sufficiently large, all relevant values remain the same with coefficients in $\F_p:$ $N(\phi,\Q) = N(\phi, \F_p),$ $\dim_{\Q} H_*(M,\Q) = \dim_{\F_p} H_*(M,\F_p),$ $\beta(\phi,\Q) = \beta(\phi,\F_p),$ and furthermore $QH_{ev}(M, \F_p)$ remains semisimple. By arguments involving the $\zp$-equivariant product-isomorphism introduced in \cite{Seidel-pants, SZhao-pants} we are able to deduce, in line with Smith theory, that in a suitable sense made precise below, the bar-length spectrum of $\phi^p$ with coefficients in $\F_p$ dominates that of $\phi$ {\em scaled by $p$}. In particular, we show the estimate \begin{equation}\label{eq: key inequality K} p \cdot \beta(\phi,\F_p) \leq K(\phi^p,\bF_p) \cdot \beta(\phi^p,\F_p),\end{equation} where $K(\phi^p,\bF_p)$ is the number of finite bars of $\phi^p$ in its coefficient $\bF_p$ barcode. By the bound in Theorem \ref{thm: bound semisimple}, the term $\beta(\phi^p, \F_p)$ is uniformly bounded, and $\beta(\phi, \F_p) = \beta(\phi, \Q)$ is a positive constant independent of $p.$ Hence $K(\phi^p, \F_p)$ and therefore $N(\phi^p, \F_p)$ must grow at least linearly in $p.$ However, this is impossible if no new $p$-periodic orbits are introduced. This finishes the argument.

\medskip

We remark that \eqref{eq: key inequality K} immediately follows from the stronger estimate \begin{equation} p \cdot \beta_{\tot}(\phi,\F_p) \leq \beta_{\tot}(\phi^p,\F_p),\end{equation} that is of independent interest, and is also key for the case of finite characteristic. It is proved as Theorem \ref{thm: beta tot} below.


\section{Preliminaries}\label{sec: prelim}

The time-one maps of isotopies $\{\phi^t_H\}_{t \in [0,1]}$ generated by time-dependent vector fields $X^t_H, \iota_{X^t_H} \omega = - d(H_t),$ are called Hamiltonian diffeomorphisms and form the group $\Ham(M,\om).$ For $H \in \cH$ we call $\overline{H},\til{H} \in \cH$ the Hamiltonians $\overline{H}(t,x) = - H(t,\phi^1_H x),$ $\til{H}(t,x) = - H(1-t,x).$ For $t \in [0,1]$ we have $\phi^t_{\overline{H}} = (\phi^t_H)^{-1},$ while the isotopy $\{\phi^t_{\til{H}}\},$ viewed as a path in $\Ham(M,\om),$ is homotopic to $\{\phi^t_{\overline{H}}\}$ with fixed endpoints. Since homotopic Hamiltonian isotopies give naturally isomorphic graded filtered Floer complexes, we shall identify the two operations $H \mapsto \overline{H},$ and $H \mapsto \til{H}.$ In particular we will identify between $H$ and the two Hamiltonians $\displaystyle\til{\overline{H}} \in \cH,$ $\displaystyle\overline{\til{H}} \in \cH.$ Similarly, for $F,G \in \cH,$ we set $F \# G \in \cH$ to generate the flow $\{\phi^t_F \phi^t_G\}_{t \in [0,1]},$ in other words $F \# G (t,x) = F(t,x) + G(t,(\phi^t_F)^{-1} x).$ A homotopic path is generated by $F \til{\#} G (t,x) = \lambda'_1(t) G(\lambda_1(t),x) + \lambda'_2(t) F(\lambda_2(t),x)$ for surjective monotone non-decreasing reparametrizations $\lambda_1,\lambda_2:[0,1] \to [0,1],$ such that $\supp \lambda'_1 < \supp \lambda'_2.$ Finally, let $\cJ(M,\om)$ be the space of $\om$-compatible almost complex structures on $M.$

\subsection{Floer theory}

Floer theory, first introduced by A. Floer \cite{Floer1,Floer2,Floer3}, is a way to set up Morse-Novikov homology for an action functional defined on a suitable cover of a loop space determined by the geometric situation at hand. We refer to \cite{OhBook} and references therein for details on the constructions described in this subsection, and to \cite{AbouzaidBook,SeidelBook, Zap:Orient}, as well as to references therein, for a discussion of canonical orientations.

\subsubsection{Hamiltonian Floer homology.} \label{subsec:abs-Ham}

Consider $H \in \cH.$ Let $\cL_{pt} M$ be the space of contractible loops in $M.$ Let $c_M: \pi_1(\cL_{pt} M) \cong \pi_2(M) \to 2 N_M \cdot \Z,$ be the surjection given by $c_M(A) = 2 \left< c_1(M,\om), A\right>.$ Let $\til{\cL}^{\min}_{pt} M = \til{\cL_{pt}} \times_{c_M} (2 N_M \cdot \Z)$ be the cover of $\cL_{pt} M$ associated to $c_M.$ The elements of $\til{\cL}^{\min}_{pt} M$ can be considered to be equivalence classes of pairs $(x,\overline{x})$ of $x \in {\cL}_{pt} M$ and its capping $\overline{x}:\D \to M,$ $\overline{x}|_{\del \D} = x.$ The symplectic action functional \[\cA_{H}: \til{\cL}^{\min}_{pt} M \to \R \] is given by \[\cA_{H}(x,\overline{x}) = \int_0^1 H(t,x(t)) - \int_{\overline{x}} \om,\] that is well-defined by monotonicity: $[\om]= \kappa \cdot c_M.$ Assuming that $H$ is non-degenerate, that is the graph $\mrm{graph}(\phi^1_H) = \{(\phi^1_H(x),x)\,|\, x \in M\}$ intersects the diagonal $\Delta_M \subset M \times M$ transversely, the generators over the base field $\bK$ of the Floer complex $CF(H;J)$ are the lifts $ \til{\cO}(H)$ to $\til{\cL}^{\min}_{pt} M$ of $1$-periodic orbits $\cO(H)$ of the Hamiltonian flow $\{\phi^t_H\}_{t \in [0,1]}.$ These are the critical points of $\cA_{H},$ and we denote by $\spec(H) = \cA_H(\til{\cO}(H))$ the set of its critical values. Choosing a generic time-dependent $\om$-compatible almost complex structure $\{J_t \in \cJ(M,\om)\}_{t \in [0,1]},$ and writing the asymptotic boundary value problem on maps $u:\R \times S^1 \to M$ defined by the negative formal gradient on $\cL_{pt} M$ of $\cA_{H},$ the count of isolated solutions with signs determined by a suitable orienation scheme, modulo $\R$-translation, gives a differential $d_{H;J}$ on the complex $CF(H;J),$ $d^2_{H;J} = 0.$ This complex is graded by the Conley-Zehnder index $CZ(x,\bar{x})$ \cite{Salamon-lectures, SalamonZehnder}, with the property that the action of the generator $A = 2N_M$ of $2N_M\cdot \Z$ has the effect $CZ(x,\bar{x} \# A) = CZ(x,\bar{x}) - 2N_M,$ and it is normalized to be equal to $n$ at a maximum of a small autonomous Morse Hamiltonian. The homology $HF_*(H)$ of this complex does not depend on the generic choice of $J.$ Moreover, considering generic families interpolating between different Hamiltonians $H,H',$ and writing the Floer continuation map, where the negative gradient depends on the $\R$-coordinate we obtain that $HF_*(H)$ in fact does not depend on $H$ either. While $CF_*(H,J)$ is finite-dimensional in each degree, it is worthwhile to consider its completion in the direction of decreasing action. In this case it becomes a free graded module of finite rank over the Novikov field $\Lambda_{M,\tmin} = \bK[q^{-1},q]]$ with $q$ being a variable of degree $(-2N_M).$ 

Moreover, for $a \in \R \setminus \spec(H)$ the subspace $CF(H,J)^{<a}$ spanned by all generators $(x,\bar{x})$ with $\cA_{H}(x,\bar{x}) < a$ forms a subcomplex with respect to $d_{H;J},$ and its homology $HF(H)^{<a}$ does not depend on $J.$ Arguing up to $\epsilon,$ one can show that a suitable continuation map sends $FH(H)^{<a}$ to $FH(H')^{<a + \cE_{+}(H-H')},$ for \[\cE_{+}(F) = \int_{0}^{1} \max_M(F_t)\,dt.\]
It shall also be useful to define $\cE_{-}(F) = \cE_{+}(-F),\; \cE(F) = \cE_{+}(F) + \cE_{-}(F).$ Moreover, for an admissible action window, that is an interval $I=(a,b),$ $a<b,$ $a,b \in \R\setminus \spec(H),$ we define the Floer complex $HF^*(H)^I$ of $H$ in this window as the homology of the quotient complex \[CF_*(H)^I = CF_*(H)^{<b}/ CF_*(H)^{<a}.\]

Finally, one can show that for each $a \in \R,$ $HF(H)^{<a}$ as well as $HF(H)^I$ for an admissible action window, depends only on the class $\til{\phi}_H$ of the path $\{\phi^t_H\}_{t \in [0,1]}$ in the universal cover $\til{\Ham}(M,\om)$ of the Hamiltonian group of $M.$

We mention that it is sometimes beneficial to consider the slightly larger covers $\til{\cL}^{\mrm{mon}}_{pt} = \til{\cL_{pt}} \times_{c_M} (2 \cdot \Z),$ $\til{\cL}^{\mrm{max}}_{pt} = \til{\cL_{pt}} \times_{c_M} \Z,$ defined via the evident inclusions $2N_M \cdot \Z \subset 2 \cdot \Z \subset \Z.$ This corresponds to extending coefficients to $\Lambda_{M,\tmon} = \bK[s^{-1},s]],$ with $\deg(s)=-2,$ and $\Lambda_{\tmon} =  \bK[t^{-1},t]],$ $\deg(t) = -1,$ respectively.

In case when $H$ is degenerate, we consider a perturbation $\cD = (K^H,J^H),$ with $K^H \in \cH,$ such that $H^{\cD} = H \# K^{H}$ is non-degenerate, and $J^H$ is generic with respect to $H^{\cD},$ and define the complex $CF(H;\cD) = CF(H^{\cD};J^H)$ generated by $\til{\cO}(H;\cD) = \til{\cO}(H^{\cD}),$ and filtered by the action functional $\cA_{H;\cD} = \cA_{H^{\cD}}.$ An admissible action window $I=(a,b)$ for $H,$ remains admissible for all $K^H$ sufficiently $C^2$-small, and the associated homology groups $HF(H;\cD)^I$ are canonically isomorphic for all $K^H$ sufficiently $C^2$-small. Hence $HF(H)^I$ is defined as the colimit of the associated indiscrete groupoid.

\subsubsection{Non-Archimedean filtrations and extension of coefficients.}\label{subsec:non-Arch}

Let $\Lambda$ be a field. A non-Archimedean valuation on $\Lambda$ is a function $\nu:\Lambda \to \R \cup \{+\infty\},$ such that  \begin{enumerate}
	\item $\nu(x) = +\infty$ if and only if $x = 0,$
	\item $l(xy) = \nu(x) + \nu(y)$ for all $x,y \in \Lambda,$
	\item $l(x+y) \geq \min\{\nu(x),\nu(y)\},$ for all $x,y \in \Lambda.$
\end{enumerate}
We set $\lambzero = \nu^{-1}([0,+\infty)) \subset \Lambda$ to be the subring of elements of non-negative valuation.  

It will sometimes be convenient to work with a larger coefficient ring in the Floer complexes. The universal Novikov field over a ground field $\bK$ is defined as \[\Lambda = \Lambda_{\tuniv} = \{\sum_j a_j T^{\lambda_j}\,|\, a_j \in \bK, \lambda_j \to +\infty \}. \]This field possesses a non-Archimedean valuation $\nu: \Lambda_{\tuniv} \to \R \cup \{+\infty\}$ given by $\nu(0) = +\infty,$ and \[\nu(\sum a_j T^{\lambda_j}) = \min\{\lambda_j\,|\,a_j \neq 0 \}.\] 
The fields $\Lambda_{M,\tmin} \subset \Lambda_{M,\tmon}$ embed into $\Lambda_{\tuniv}$ via $s \mapsto T^{2\kappa_M}.$  This lets us pull back the valuation on $\Lambda_{\tuniv}$ to a valuation on $\Lambda_{M,\tmin},$ and $\Lambda_{M,\tmon}.$ We use the following convention throughout the paper: whenever the ground field of $\Lambda$ is clear from the context, we omit it. In particular, we shall often say that an operator is defined over $\lamzero,$ keeping the ground field implicit.

Now let $\Lambda$ be a field with non-Archimedean valuation $\nu.$ Following \cite{UsherZhang}, given a finite dimensional $\Lambda$-module $C,$ we call a function $l:C \to \R \cup \{-\infty\}$ a non-Archimedean filtration (function), if it satisfies the following properties: \begin{enumerate}
	\item $l(x) = -\infty$ if and only if $x = 0,$
	\item $l(\lambda x) = l(x) - \nu(\lambda)$ for all $\lambda \in \Lambda, x \in C,$
	\item \label{prop:maximu} $l(x+y) \leq \max\{l(x),l(y)\},$ for all $x,y \in C.$
\end{enumerate}

It is easy to see \cite[Proposition 2.1]{EntovPolterovichCalabiQM}, \cite[Proposition 2.3]{UsherZhang} that the maximum property \eqref{prop:maximu} implies that whenever $l(x) \neq l(y),$ one has in fact \begin{equation}\label{eq:max property filtration} l(x+y) = \max\{l(x),l(y)\}.\end{equation} A $\Lambda$-basis $(x_1,\ldots,x_N)$ of $(C,l)$ is called {\em orthogonal} if \[l(\sum \lambda_j x_j) = \max \{l(x_j) -\nu(\lambda_j) \} \] for all $\lambda_j \in \Lambda.$ It is called {\em orthonormal} if in addition $l(x_j) = 0$ for all $j.$ At this point, we note that a linear transformation $T:C \to C$ with matrix $P \in GL(N,\lambzero)$ in an orthonormal basis satisfies $T^*l = l.$ In particular it sends each orthogonal, respectively orthonormal, basis to an orthogonal, respectively orthonormal basis. 


Consider the Floer complex from Section \ref{subsec:abs-Ham},  as a finite-dimensional $\Lambda$-module $C,$ for a suitable Novikov field $\Lambda.$ The function $\cA:C \to \R \cup \{-\infty\}$ given by $\cA(x) = \inf\{a\,|\, x \in C^{<a}\}$ is a non-Archimedean filtration. It can be computed as follows. Consider a standard basis $x_1,\ldots,x_N$ of $C$ over $\Lambda,$ consisting of arbitrarily chosen lifts of the finite set of periodic orbits involved. Then we have \begin{equation}\cA(\sum \lambda_j x_j) = \max\{\cA(x_j) - \nu(\lambda_j)\}\end{equation} for all $\lambda_j \in \Lambda.$ In other words $x_1,\ldots,x_N$ is an orthogonal basis for $(C,\cA).$ Finally, we note that $d^*\cA \leq \cA,$ and in fact for each $x \in C \setminus \{0\}$ the strict inequality $\cA(d(x)) < \cA(x)$ holds. We say in the latter case that the filtered complex is {\em strict}.

To extend coefficients in $C,$ we take \[\overline{C} = C \otimes_{\Lambda} \Lambda_{\tuniv}\] and define a non-Archimedean filtration function $\cA:\overline{C} \to \R\cup \{-\infty\}$ on $\overline{C}$ by declaring that $x_1 \otimes 1,\ldots,x_N \otimes 1$ is an orthogonal basis for $(\overline{C},\cA).$ Finally, we note that the basis $(\overline{x}_1,\ldots,\overline{x}_N) = (T^{\cA(x_1)} x_1, \ldots, T^{\cA(x_N)} x_N)$ is an orthonormal basis of $(\overline{C},\cA)$ that is canonical, in the sense that it does not depend on the ambuguity in the choice of $x_1,\ldots,x_N.$

\subsection{Quantum homology and the PSS isomorphism}\label{subsec:QH}

In this section we describe the quantum homology of a symplectic manifold. It may be helpful to think of it as the Hamiltonian Floer homology, when the Hamiltonian is given by a $C^2$-small, time-independent Morse function. Alternatively, one can consider it as the cascade approach \cite{Frauenfelder} to Morse homology for the unperturbed symplectic area functional on the space $\til{\cL}^{\min}_{pt} M.$ For further information on these subjects we refer for example to \cite{SeidelBook,OhBook,LeclercqZapolsky}.

\subsubsection{Quantum homology}

Fixing a ground field $\bK,$ set $QH(M) = QH(M,\bK) = H_*(M;\Lambda_{M,\tmin}),$ as a $\Lambda_{M,\tmin}$-module. This module has the structure of a graded-commutative unital algebra over $\Lambda_{M,\tmin}$ whose product, deforming the classical intersection product on homology, is defined in terms of $3$-point genus $0$ Gromov-Witten invariants \cite{McDuffSalamon-BIG,Liu-assoc,RuanTian-qh1,RuanTian-qh2,Witten-2d}. The unit for this {\em quantum product} is the fundamental class $[M]$ of $M,$ as in the case of the classical homology algebra. The non-Archimedean filtration $\cA:QH(M) \to \R \cup \{-\infty\}$ is given by declaring $E \otimes 1_{\Lambda},$ for a basis $E$ of $H_*(M,\bK)$ to be an orthonormal basis for $(QH(M),\cA).$ Finally, the elements of even degree form a genuinely commutative subalgebra $QH_{ev}(M)$ of $QH(M).$ Finally, as a $\Lambda$-module, it will be convenient to consider $QH(M)$ as the homology of a Morse complex $CM(f,\Lambda; \cl{D}),$ of a Morse function $f,$ and perturbation datum given by a Riemannian metric $\rho,$ such that the pair $(f,\rho)$ is Morse-Smale. We shall omit $\cl{D}$ from the notation.

\subsubsection{Floer homology as a module over quantum homology}

In the absolute case, as discussed in detail in \cite{PolSheSto}, an element $\alpha_M \in QH_m(M) \setminus \{0\}$ gives, for $H \in \cH,$ and $r \in \Z, a \in \R$ a map \[(\alpha_M\ast):HF_{r}(H)^a \to HF_{r+m-2n}(H)^{a+\cA(\alpha_M)}.\] It is in fact a morphism \[(\alpha_M\ast):V_{r}(H) \to V_{r+m-2n}(H)[\cA(\alpha_M)]\] of persistence modules, as defined in Section \ref{subsec: Floer persistence}. This morphism is constructed, in a manner very similar to the quantum cap product (see \cite[Example A.4]{PSS} or \cite{SeidelMCG,Schwarz:action-spectrum,Floer3}) by counting negative $\rho$-gradient trajectories $\gamma:(-\infty,0] \to M$ of a Morse function $f$ on $M,$ for a generic pair $(f,\rho),$ asymptotic at $s \to -\infty$ to critical points of $f,$ and having $\gamma(0)$ incident to Floer cylinders $u:\R \times S^1 \to M$ at $u(0,0).$

\subsubsection{Piunikhin-Salamon-Schwarz isomorphisms}

With the conventions on the Conley-Zehnder index that we use, one obtains a map $PSS: QH_*(M) \to HF_{*-n}(H)$ by counting (for generic auxiliary data) isolated configurations of negative gradient trajectories $\gamma:(-\infty,0] \to M$ incident at $\gamma(0)$ with the asymptotic of $\lim_{s \to -\infty} u(s,-),$ as $s \to -\infty$ of a map $u: \R \times S^1 \to M,$ satisfying a Floer equation \[\del_s \,u + J_t(u) \,(\del_t \,u - X^t_{K}(u)) = 0, \] where for $(s,t) \in \R \times S^1,$ $K(s,t) \in \sm{M,\R}$ is a small perturbation of $\beta(s) H_t,$ coinciding with it for $s \ll -1$ and $s \gg +1,$ and $\beta: \R \to [0,1]$ is a smooth function satisfying $\beta(s) \equiv 0$ for $s \ll -1$ and $\beta(s) \equiv 0$ for $s \gg +1.$ This so-called Piunikhin-Salamon-Schwarz map \cite{PSS} is an isomorphism of $\Lambda_{M,\tmin}$-modules, which in fact intertwines the quantum product on $QH(M)$ with the pair of pants product in Hamiltonian Floer homology.

Finally, it is worthwhile to note that the quantum product map \[QH(M) \otimes QH(M) \to QH(M), \] is isomorphic to the module action maps \[QH(M) \otimes HF(H) \to HF(H), \] via the isomorphism $\id \otimes PSS$ on the left hand side, and $PSS$ on the right hand side. 

\subsection{Spectral invariants}\label{subsec:spec}

Given a filtered complex $(C,\cA),$ to each homology class $\alpha \in H(C)$, denoting by $H(C)^{<a} = H(C^{<a}),$ $C^{<a} = \cA^{-1} (-\infty,a),$ we define
a spectral invariant by \[c(\alpha, (C, \cA)) = \inf\{a \in \R\,|\, \alpha \in \ima(H(C)^{<a} \to H(C)) \}.\] For $(C,\cA) = (CF(H;\cD),\cA_{H ;\cD})$ we denote $c(\alpha, H; \cD) = c(\alpha,(C,\cA)).$ Furthermore, one can one can obtain classes $\alpha$ in the Hamiltonian Floer homology by the PSS isomorphism. This lets us define spectral invariants by: 
\[c(\alpha_M, H; \cD) = c(PSS(\alpha_M),(CF(H;\cD),\cA_{H;\cD})),\] for $\alpha_M \in QH(M).$ 
From the definition it is clear that the spectral invariants do not depend on the almost complex structure term in $\cD.$ Moreover, if $H$ is non-degenerate, we may choose the Hamiltonian term in $\cD$ to vanish identically, and denote the resulting invariants by $c(-,H).$ Moreover, by \cite[Section 5.4]{BiranCorneaRigidityUniruling} spectral invariants remain the same under extension of coefficients, hence below we do not specify the Novikov field $\Lambda$ that we work over. Spectral invariants enjoy numerous useful properties, all of which hold for monotone symplectic manifolds. We summarize the relevant ones below:

\begin{enumerate}
	\item {\em spectrality:} for each $\alpha_M \in QH(M) \setminus \{0\},$ and $H \in \cH,$ \[c(\alpha_M, H) \in \Spec(H).\]
	\item {\em non-Archimedean property:} $c(-,H;\cD)$ is a non-Archimedean filtration function on $QH(M),$ as a module over the Novikov field $\Lambda$ with its natural valuation. 
	\item {\em continuity:} for each $\alpha_M \in QH(M) \setminus \{0\},$ and $F,G \in \cH,$
	\[|c(\alpha_M,F) - c(\alpha_M,G)| \leq \cE(F-G),\] 
	\item {\em triangle inquequality:} for each $\alpha_M,\alpha'_M \in QH(M),$ and $F,G \in \cH,$
	\[c(\alpha_M \ast \alpha'_M,F\#G) \leq c(\alpha_M ,F) + c(\alpha'_M ,G),\]
		
\end{enumerate}

We remark that by the continuity property, the spectral invariants are indeed defined for all $H \in \cH$ and all the properties above apply in this generality.

In the abstract case of a complex $(C,d)$ over $\Lambda$ filtered by $\cA,$ the non-Archimedean property lets us consider the spectral invariant map, as an {\em induced non-Archimedean filtration function} \[\mrm{H}(\cA) : H(C,d) \to \R \cup \{-\infty\}.\] 

We remark that the key part of the non-Archimedean property, the maximum property, is called the {\em characteristic exponent} property in \cite{EntovPolterovichCalabiQM}. 

\subsubsection{Spectral norm.}\label{subsubsec:spec norm}
For $H \in \cH$ we define its spectral pseudo-norm by \[\gamma(H) = c([M],H) + c([M],\overline{H}),\] which depends only on $\til{\phi}_H,$ and by a result of \cite{Viterbo-specGF,Schwarz:action-spectrum, Oh-specnorm} (see also \cite{Usher-sharp,McDuffSalamon-BIG}) gives the following non-degenerate spectral norm $\gamma:\Ham(M,\om) \to \R_{\geq 0},$ \[\gamma(\phi) = \inf_{\phi^1_H = \phi} \gamma(H),\] and hence the bi-invariant spectral distance $\gamma(\phi,\phi') = \gamma(\phi' \phi^{-1}).$ 

\subsubsection{Mean-index}\label{subsubsec: mean-index}

We require the notion of mean-index, introduced in symplectic topology in \cite{SalamonZehnder}. For a Hamiltonian $H$ generating $\til{\phi} \in \til{\Ham}(M,\om)$ and capped periodic orbit $\ol{x}$ of $H,$ we set \[\Delta(H,x) = \Delta(\til{\phi}_H,x) = \lim_{k \to \infty} \frac{1}{k} CZ(\til{\phi}^k, \ol{x}^{(k)}),\] where $\ol{x}^{(k)}$ is $\ol{x}$ iterated $k$ times, which is indeed a capped periodic orbit of a Hamiltonian generating $\til{\phi}^k.$ The limit exists, since the Conley-Zehnder index comes from a quasi-morphism $\til{Sp}(2n,\R) \to \R$ (see \cite{EntovPolterovich-rigid}).

The main properties of the mean index that we use are as follows:

\begin{enumerate}
	\item {\em homogeneity:} $\Delta(\til{\phi}^k,\ol{x}^{(k)}) = k \cdot \Delta(\til{\phi},\ol{x}),$ for all $k \in \Z_{>0}.$
	\item {\em recapping:} $\Delta(\til{\phi},\ol{x} \# A) = \Delta(\til{\phi},\ol{x}) - 2N_M.$
	\item {\em distance to index:} $CZ(\til{\phi},\ol{x}) \in [\Delta(\til{\phi},\ol{x}) - n, \Delta(\til{\phi},\ol{x}) + n]$ 
	\item {\em support of local Floer homology:} $HF^{\loc}_r(\til{\phi},\ol{x}) = 0,$ unless $r \in [\Delta(\til{\phi},\ol{x}) - n, \Delta(\til{\phi},\ol{x}) + n].$
\end{enumerate}

\subsection{Floer persistence}\label{subsec: Floer persistence}

\subsubsection{Rudiments of persistence modules}

Let $\mrm{Vect}_{\bK}$ denote the category of finite-dimensional vector spaces over $\bK,$ and $(\R,\leq)$ denote the poset category of $\R.$ A {\em persistence module} over $\bK$ is a functor \[V:(\R,\leq) \to \mrm{Vect}_{\bK}.\] In other words $V$ consists of a collection $\{V^a \in \mrm{Vect}_{\bK}\}_{a \in \R}$ and $\bK$-linear maps $\pi_V^{a,a'}:V^a \to V^{a'}$ for each $a \leq a',$ satisfying $\pi_V^{a,a} = \id_{V^a},$ and $\pi_V^{a',a''} \circ \pi_V^{a,a'} = \pi_V^{a,a''}$ for all $a\leq a'\leq a''.$ These functors with their natural transformations form an abelian category \[Fun((\R,\leq),\mrm{Vect}_{\bK}),\] where $A \in \hom(V,W)$ consists of a collection $\{A^a \in \hom_{\bK}(V^a,W^a)\}_{a \in \R}$ that commutes with the maps $\pi_{V}^{a,a'}, \pi_{V}^{a,a'},$ for each $a \leq a'.$ We require the following further technical assumptions, that hold in all our examples: \begin{enumerate}
	\item {\em support:} $V^a = 0$ for all $a \ll 0.$
	\item {\em finiteness:} there exists a finite subset $S \subset \R,$ such that for all $a,a'$ in each connected component of $\R \setminus S,$ the map $\pi_V^{a,a'}:V^a \to V^{a'}$ is an isomorphism.
	\item {\em continuity:} for each two consecutive elements $s_1< s_2$ of $S,$ and $a \in (s_1,s_2),$ the map $\pi^{a,s_2}: V^a \to V^{s_2}$ is an isomorphism.
\end{enumerate}

Persistence modules with these properties form a full abelian subcategory \[\pemod \subset Fun((\R,\leq),(\mrm{Vect}_{\bK})).\] We set $V^{\infty} = \lim_{a \to \infty} V^a.$

The {\em normal form theorem} \cite{CarlZom,CrawBo} for persistence modules states that the isomorphism class of $V \in \pemod$ is classified by a finite multiset $\cB(V) = \{(I_k,m_k)\}_{1 \leq k \leq N'}$ of intervals $I_k \subset \R,$ where $I_k = (a_k,b_k)$ for $k \in (0,K] \cap \Z,$ and $I_k = (a_k,\infty) $ for $k \in (K,N'] \cap \Z$ for some $0 \leq K= K(V) \leq N'.$ We denote $B=B(V) = N'-K \geq 0.$  The intervals are called bars, and a multiset of bars is called a barcode. The bar lengths are defined as $|(a_k,b_k)| = b_k - a_k,$ and $|(a_k,\infty)| = +\infty.$

The {\em isometry theorem} for persistence modules \cite{CdSGO-structure,BauLes,CCSGGO-proximity}, culminating the active development initiated in \cite{CEH-stability},  states the fact that the barcode map \[\cB:(\pemod,d_{\mrm{inter}}) \to (\barc,d_{\mrm{bottle}})\] \[V \mapsto \cB(V)\] is {\em isometric} for the following two distances. 

The {\em interleaving distance} between $V,W \in \pemod$ is given by \begin{align*}d_{\mrm{inter}}(V,W) = \inf \{\delta > 0\;|\; \exists &f \in \hom(V,W[\delta]), g \in \hom(W,V[\delta]), \\ &g[\delta]\circ f = sh_{2\delta,V}, f[\delta]\circ g = sh_{2\delta,W} \},\end{align*} where for $V \in \pemod,$ and $c \in \R,$ $V[c] \in \pemod$ is defined by pre-composition with the functor $T_c: (\R,\leq) \to (\R,\leq), t\to t+c,$ and for $c \geq 0,$ $sh_{c,V} \in \hom (V,V[c])$ is given by the natural transformation $\id_{(\R,\leq)} \to T_c.$ The pair $f,g$ from the definition is called a {\em $\delta$-interleaving}. A priori, $d_{\mrm{inter}}(V,W) \in \R_{\geq 0} \cup \{ \infty \},$ and it is finite if and only if $V^{\infty} \cong W^{\infty}.$   

The {\em bottleneck distance} between $\cB,\cC \in \barc$ is given by \[d_{\mrm{bottle}}(\cB,\cC) = \inf \bra{\delta > 0\,|\;\exists\; \delta-\text{matching between} \;\cB,\cC },\] where a $\delta$-matching between $\cB,\cC$ is a bijection $\sigma:\cB^{2\delta} \to \cC^{2\delta}$ between two sub-multisets $\cB^{2\delta} \subset \cB,$ $\cC^{2\delta} \subset \cC,$ each containing all the bars of length $> 2\delta$ of $\cB,$ $\cC$ respectively, such that if $\sigma((a,b)) = (a',b')$ then $|a-a'|\leq \delta,$ $|b-b'| \leq \delta.$ Observe that $d_{\mrm{bottle}}(\cB,\cC) \in \R_{\geq 0} \cup \{ \infty \}$ if and only if $\cB, \cC$ have the same number of infinite bars.

Finally, we record the quotient space $(\barc',d'_{\mrm{bottle}})$  of $(\barc,d_{\mrm{bottle}})$ by the isometric $\R$-action by shifts: $c \in \R$ acts by $\cB = \{(I_k,m_k)\} \mapsto \cB[c] = \{(I_k - c, m_k)\},$ and $d'_{\mrm{bottle}}([\cB],[\cC]) = \inf_{c\in \R} d'_{\mrm{bottle}}(\cB,\cC[c])$ for $\cB,\cC \in \barc.$ Note that bar-lengths give a well-defined map from $\barc'$ to multi-subsets of $\R_{>0} \cup \bra{+\infty}.$

\subsubsection{Floer persistence: interleaving, invariance, spectral norms}

As remarked in Section \ref{subsec:abs-Ham} for each $r \in \Z,$ the degree $r$ subspace $C_r$ of the Floer complex $C$ considered therein is finite-dimensional over the base field $\bK.$ This implies that the degree $r$ homology $H_r(C)^{<a} = H_r(C^{<a})$ is in $\mrm{Vect}_{\bK}$ for all $a \in \R.$ Furthermore, inclusions $C^{<a} \to C^{<a'}$ of graded complexes for $a \leq a',$ yield maps $\pi^{a,a'}:H_r(C)^{<a} \to H_r(C)^{<a'}.$ As it was first observed in \cite{PolShe} (see also \cite{PolSheSto}), the collection $V_r(C)$ of the vector spaces $\{H_r(C)^a\}_{a\in \R},$ and maps $\{\pi^{a,a'}\},$ constitutes an object of the category $\pemod.$ We will denote this persistence module by $V_r(H;\cD),$ in general, and by $V_r(H),$ when $H$ is non-degenerate. We use these notations interchangeably, with the understanding that the former is used in the degenerate case, where the Hamiltonian terms in the perturbation data is considered to be as $C^2$-small as necessary, and the latter is used in the non-degenerate case.

Finally, by \cite{PolShe,PolSheSto} Floer continuation maps induce $\cE(H-H')$-interleavings between the pairs $V_r(H),V_r(H').$ This implies that \begin{equation}\label{eq: barc Hofer Lip} d_{int}(V_r(H),V_r(H')) \leq \til{d}_{\mrm{Hofer}}(\til{\phi}_H,\til{\phi}_{H'}) \end{equation} where $\til{d}_{\mrm{Hofer}}$ is the Hofer pseudo-metric on $\til{\Ham}(M,\om)$ defined by \[\til{d}_{\mrm{Hofer}}(\til{\phi}_{H},\til{\phi}_{H'}) = \inf \cE(F-G),\] the infimum running over all $F,G \in \cH$ with $\til{\phi}_{F} = \til{\phi}_{H},$ $\til{\phi}_{G} = \til{\phi}_{H'}.$

We recall the ring $\Lambda_{\tmon} = \bK[t^{-1},t]]$ with variable $t$ of degree $(-1).$ For $H \in \cH,$  consider the associated Floer persistence module $V_0(H)$ of degree $0$ with coefficients in $\Lambda_{\tmon}.$ Let $\cB_0(H)$ be the corresponding barcode.  By \cite{KS-bounds} (see also \cite[Propositions 5.3, 6.2]{UsherBD2}) the image $\cB'_0(\phi^1_H)$ of this barcode in $(\barc',d'_{\mrm{bottle}})$ depend only on $\phi^1_H.$ Finally, by a change of coordinates given by $\psi \in \Symp(M,\om),$ there is an identity of barcodes \begin{equation}\label{eq:conj invariance barcodes}\cB'(\phi) = \cB'(\psi \phi \psi^{-1})\end{equation}

We define the bar-length spectrum of $\phi^1_H$ to coincide with the corresponding sub-multiset of $\R_{>0} \cup \bra{+\infty}$ arranged as an increasing sequence, taking into account multiplicities. We define the {\em boundary depth} $\beta(\phi)$ of $\phi$ to be the {\em maximal length of a finite bar} in the associated barcode. This notion was first introduced by Usher \cite{UsherBD1,UsherBD2} in different terms, and shown to satisfy various properties, including the above invariance statement.

Finally, in view of \eqref{eq: barc Hofer Lip}, we obtain for $\phi, \psi \in \Ham(M,\om),$ \begin{equation}\label{eq: barc' Hofer Lip}  d'_{\mrm{bottle}}(\cB'(\phi),\cB'(\psi)) \leq {d}_{\mrm{Hofer}}(\phi,\psi)\end{equation} where \[d_{\mrm{Hofer}}(\phi,\psi) = \inf_{\phi^1_F = \phi, \phi^1_G = \psi} \til{d}(\til{\phi}_{F},\til{\phi}_{G})\] is the celebrated Hofer metric \cite{HoferMetric,Lalonde-McDuff-Energy} on $\Ham(M,\om).$ The method of filtered continuation elements introduced in \cite{KS-bounds}, with inspiration from \cite{BiranCorneaS-Fukaya,AK-simplehomotopy}, was used to improve \eqref{eq: barc' Hofer Lip} to \begin{equation}\label{eq: barc' spectral Lip}
d'_{\mrm{bottle}}(\cB'(\phi),\cB'(\psi)) \leq \frac{1}{2}{\gamma}(\phi,\psi). 
\end{equation}
This was also extended to the relative setting therein. Proposition \ref{prop: beta leq gamma} below is yet another extension of this result.

\subsubsection{Barcode of Hamiltonian diffeomorphism with isolated fixed points}\label{subsubsec: barcode isolated}
Furthermore, it was shown in \cite{S-Zoll} following \cite{LSV-conj}, that if $\phi$ has isolated fixed points, then the barcode $\cB'(\phi)$ consists of a finite number of bars, of them $B(\bK) = \dim_{\bK} H_*(M,\bK)$ are infinite, and $K(\phi,\bK)$ are finite, where $N(\phi, \bK)$ as defined in \eqref{eq: homological count N} satisfies $N(\phi,\bK) = 2K(\phi, \bK) + B(\bK).$ Furthermore, for a sequence $\phi_j$ of $C^2$-small Hamiltonian perturbations of $\phi,$ where the norm of the perturbation goes to zero, $\cB'(\phi_j) \to \cB'(\phi)$ in $d'_{\mrm{bottle}}.$ Finally, the bar-lengths are given by differences $\cA_H(\ol{x}) - \cA_H(\ol{y})$ of points in $\spec(H),$ with mean-index $\Delta(H, \ol{x}) \in [-2n, 2N_M+2n),$ $\Delta(H, \ol{y}) \in [-2n-1, 2N_M+2n-1),$  for $H \in \cH$ generating $\phi$ (see Section \ref{subsubsec: mean-index} for the definition of the mean-index, and the support of local Floer homology in particular). Therefore set of possible bar-lengths is finite, and hence there is a constant $\epsilon_1$ such that each bar-length $\beta,$ counted {\em without} multiplicity, is the unique bar-length in the interval $(\beta-\epsilon_1, \beta+\epsilon_1) \subset \R_{>0}.$

We conclude this section by observing that \eqref{eq: barc' Hofer Lip} implies that the boundary depth $\beta(\phi)$ is $1$-Lipschitz in the Hamiltonian spectral norm. In particular $\beta(\phi),$ similarly to $\gamma(\phi),$ is defined for arbitrary $\phi \in \Ham(M,\om).$ 

\subsubsection{Bar-lengths, extended coefficients, and torsion exponents.}\label{subsubsec: bar-length three flavors}

We note the following two alternative descriptions of the bar-length spectrum. Firstly, consider each one of the relevant Floer complexes $(C,d)$ over $\Lambda = \Lambda_{\tmin},$ filtered by $\cA$ as in Section \ref{subsec:non-Arch}. By \cite{UsherZhang}, the complex $(C,d)$ admits an orthogonal basis \[E = (\xi_1,\ldots,\xi_{B},\eta_1,\ldots,\eta_K,\zeta_1,\ldots,\zeta_K)\] such that $d\xi_j = 0$ for all $j \in (0,B] \cap \Z,$ and $d \zeta_j = \eta_j$ for all $j \in (0,K] \cap \Z.$ The finite bar-lengths are then given by $\bra{\beta_j = \beta_j(C,d) = \cA(\zeta_j) - \cA(\eta_j)}$ for $j \in (0,K] \cap \Z,$ which we assume to be arranged in increasing order, while there are $B$ infinite bar-lengths, corresponding to $\xi_j$ for $j \in (0,B] \cap \Z.$ We note that this description yields the identity $N = B + 2K,$ where the numbers $N,B,K$ can be computed via $N = \dim_{\Lambda} C,$ $B = \dim_{\Lambda} H(C,d),$ and $K = \dim \ima(d).$ Furthermore, from this description it is evident that for a field extension $\mathbb{L}$ of $\mathbb{K},$ the bar-length spectrum with ground field $\mathbb{L}$ is identical to that with ground field $\bK.$

Extending coefficients to $\Lambda_{\tuniv},$ we can further normalize the basis $E$ to obtain the orthonormal basis \begin{align*}\overline{E} &= (\bar{\xi}_1,\ldots,\bar{\xi}_{B},\bar{\eta}_1,\ldots,\bar{\eta}_K,\bar{\zeta}_1,\ldots,\bar{\zeta}_K) = \\ &= (T^{\cA(\xi_1)}\xi_1,\ldots,T^{\cA(\xi_B)}\xi_{B},T^{\cA(\eta_1)}\eta_1,\ldots,T^{\cA(\eta_K)}\eta_K,T^{\cA(\zeta_1)}\zeta_1,\ldots,T^{\cA(\zeta_K)}\zeta_K).\end{align*} This basis satisfies $d\bar{\xi}_j = 0$ for all $j \in (0,B] \cap \Z,$ and $d \bar{\zeta}_j = T^{\beta_j} \bar{\eta}_j$ for all $j \in (0,K] \cap \Z.$ We note that passing to this basis has the following computational advantage. First, each two orthonormal bases are related by a linear transformation $T:\overline{C} \to \overline{C}$ with matrix in $GL(N,\Lambda^0_{\tuniv}).$ Second, to compute the bar-length spectrum, it is sufficient to consider the matrix $[d]$ of $d:\overline{C} \to \overline{C}$ in any orthonormal basis, for example the canonical one from Section \ref{subsec:non-Arch}, and bring it to Smith normal form over $\Lambda^0_{\tuniv}.$ The diagonal coefficients, in order of increasing valuations, will be $\{T^{\beta_j}\}.$ We remark that while $\Lambda^0_{\tuniv}$ is not a principal ideal domain, each of its finitely generated ideals is indeed principal, and therefore Smith normal form applies in this case. 

It shall be important to remark that the considerations regarding the Smith normal form and the bar-length spectrum apply to the case of arbitrary complexes $(C,d)$ over $\Lambda$ with non-Archimedean filtration function $\cA.$ If $d^* \cA \leq \cA,$ where the inequality is not necessarily strict for non-zero elements, then the bar-length spectrum may contain a few entries $\beta_j = 0.$ Following \cite{UsherZhang}, we shall call this bar-length spectrum {\em verbose}, from which the {\em concise} bar-length spectrum is obtained by recording only the strictly positive entries.

In fact given a filtered map $D:(C,\cA) \to (C',\cA')$ between two filtered $\Lambda$-modules, that is $D^* \cA' \leq \cA,$ it is shown in \cite{UsherZhang} that $D$ has a {\em non-Archimedean spectral value decomposition}: that is orthogonal bases $E = E_{\mrm{coim}} \sqcup E_{\ker}$ of $(C,\cA),$ and $E' = E_{\ima} \sqcup E_{\mrm{coker}}$ of $(C',\cA')$ such that $D(E_{\ker}) = 0,$ while $D|_{E_{\mrm {coim}}}:E_{\mrm{coim}} \xrightarrow{\sim} E_{\mrm{im}}$ is an isomorphism of sets. The spectral values consist of the numbers $\beta_e = \cA(e) - \cA'(D(e))\geq 0$ for $e \in E.$ Over $\Lambda = \Lambda_{\tuniv},$ we may instead ask for $E,E'$ to be orthonormal, and require that for each $e \in E_{\mrm{coim}},$ there exists $e' \in E_{\mrm{im}},$ and $\beta_e \geq 0,$ such that $D(e) = T^{\beta_e} e',$ and $D(E_{\ker}) = 0.$ It is easy to see that the spectral values correspond directly to the bar-lengths for the filtered complex \[(Cone(D),\cA \oplus \cA'),\] given as a $\Lambda$-module by $C \oplus C',$ with filtration $\cA \oplus \cA' = \max\{\cA,\cA'\},$ and differential \[d_{Cone}(c,c') = (d_C(c), D(c) - d_{C'}(c') ).\]

Finally, following \cite{FOOO-polydiscs}, it is easy to see that the matrix of the differential $d$ in the canonical basis from Section \ref{subsec:non-Arch} has all coefficients in $\lambzero_{\tuniv}.$ Indeed, the coefficient of $\overline{x_i}$ in $d\overline{x}_j$ is given by \[\left<d\bar{x}_j, x_i\right> = \sum \epsilon(u) T^{E(u)} \in \lambzero_{\tuniv},\] where $E(u)$ is the energy of $u$ as a negative gradient trajectory of the corresponding action functional, $\epsilon(u)$ is a sign determined by a suitable orientation scheme, and the sum runs over all isolated (modulo $\R$-translations) negative gradient trajectories asymptotic to $x_j$ at times $s \to -\infty,$ and to $x_i$ at times $s \to +\infty.$ Therefore, one can define the Floer complex from Section \ref{subsec:abs-Ham} with coefficients in $\lambzero_{\tuniv}.$ Its homology will be a finitely generated $\lambzero_{\tuniv}$-module, and will therefore have the form $\mathcal{F} \oplus \cl{T},$ where $\cl F$ is a free $\lambzero_{\tuniv}$-module, and $\cl T$ is a torsion $\lambzero_{\tuniv}$-module. The bar-lengths in this setting are given by the identity \[\cl T \cong \bigoplus_{1 \leq j \leq K}\lambzero_{\tuniv}/(T^{\beta_j}).\] We summarize the above discussion as follows.

\medskip
\begin{lma}\label{lemma: different barcodes}
The three definitions of the bar-length spectrum for Hamiltonian Floer homology on monotone symplectic manifolds, due to Fukaya-Oh-Ohta-Ono \cite{FOOO-polydiscs}, Polterovich-Shelukhin \cite{PolShe} (see also \cite{PolSheSto, KS-bounds}), and Usher-Zhang \cite{UsherZhang}, respectively, coincide. 
\end{lma}

We refer to \cite{UsherZhang,KS-bounds, S-Zoll} for more details of the identifications between the various descriptions of the bar-length spectrum, and notions of persistence.

\subsubsection{$\delta$-quasi-equivalences}\label{sec: quasi-eq}

Finally, we shall require a version of the notion of $\delta$-interleaving in the $\lamzero$ setting. We call two $\lamzero$-complexes $(C,d),$ $(C',d')$ $\delta$-quasi-equivalent if there exist $\lamzero$-chain maps \begin{align*}F&: (C,d) \to (C',d'),\\ G &: (C',d') \to (C,d),\end{align*} and $\lamzero$-chain homotopies \begin{align*}H&: (C,d) \to (C,d),\\ H'&: (C',d') \to (C',d'),\end{align*} such that \begin{align} G\circ F - T^\delta \id_C &= d H + H d\\ \notag F\circ G - T^\delta \id_{C'} &= d H' + H' d.\end{align} In this case $C \otimes \Lambda,$ $C' \otimes \Lambda$ with the natural induced filtration are $\delta/2$-quasi-equivalent in the sense of Usher-Zhang \cite{UsherZhang}: the quasi-equivalences are then given by $T^{-\delta/2} F, T^{-\delta/2} G,$ and homotopies by $T^{-\delta} H, T^{-\delta} H'.$ In this case it follows immediately from \cite[Theorem 1.4]{UsherZhang} that the bar-length spectra \[\{\beta_j \}_{1 \leq j \leq K+B}, \{\beta'_j \}_{1 \leq j \leq K'+B'}\] of $(C,d), (C',d')$ satisfy: $\beta_{B+K-j} = + \infty$ if and only if $\beta'_{B'+K'-j} = + \infty,$ in other words $B = B'.$ Moreover, $\beta_{K-j} > 4\delta$ if and only if $\beta'_{K'-j} > 4 \delta.$ Finally, for such $j\geq 0,$ \[|\beta_{K-j} - \beta_{K'-j}| < 2\delta.\] We say that these two bar-length spectra are {\em $2\delta$-close}.

In addition we remark that a $\delta_1$-quasi-equivalence $F_1, G_1,$ with homotopies $H_1, H'_1$ between $(C_0,d_0),$ and $(C_1,d_1)$ and a $\delta_2$-quasi-equivalence $F_2, G_2,$ with homotopies $H_2, H'_2$ between $(C_1,d_1),$ and $(C_2,d_2)$ give rise to a ${(\delta_1+\delta_2)}$-quasi-equivalence $F=F_2 \circ F_1,$ $G = G_1 \circ G_2,$ between $(C_0,d_0),$ and $(C_2,d_2),$ with suitable homotopies $H,H'$ given in terms of the above quasi-equivalences. 

Furthermore, $(C,d)$ and $(C,T^{\delta}d)$ for $\delta \geq 0$ are $\delta$-quasi-equivalent. This is easy to see by choosing a $\lamzero$-basis $\{\ol{\xi}_i, \ol{\eta}_j, \ol{\zeta}_j\}_{1 \leq i \leq B,\; 1 \leq j \leq K}$ for $C$ such that $d\ol{\xi}_i = 0,$ $d\ol{\zeta}_j = T^{\beta_j} \ol{\eta}_j$ for all $1 \leq i \leq B,$ $1 \leq j \leq K.$ Then $F:(C,d) \to (C,T^{\delta}d),$ $G:(C,T^{\delta}d) \to (C,d),$ forming a $\delta$-quasi-equivalence, are given by $F(\ol{\xi}_i) = T^{\delta/2} \ol{\xi}_i, F(\ol{\eta}_j) = T^{\delta} \ol{\eta}_j, F(\ol{\zeta}_j) = \ol{\zeta}_j$ and $G(\ol{\xi}_i) = T^{\delta/2} \ol{\xi}_i, G(\ol{\eta}_j) = \ol{\eta}_j, G(\ol{\zeta}_j) = T^{\delta} \ol{\zeta}_j$ for all $1 \leq i \leq B,$ $1 \leq j \leq K.$ In this case we can take $H = 0,$ $H' = 0.$

\subsubsection{Operations and $\lamzero$}\label{subsubsec: operations and lambzero}
In a related direction, we observe that when a Floer-homological operation, arising from a moduli space of marked Riemann surfaces, does {\em not} have zero curvature, which is for example the case of a Floer continuation map that changes the Hamiltonian function, there is a subtlety in defining this operation over $\lamzero.$ The subtlety appears since for compositions to function properly, we need to consider the {\em topological energy} of the curves we count, that is, the difference of actions (see \cite[Section (8g)]{SeidelBook}, \cite[Section 3.3]{BiranCorneaS-Fukaya}, or \cite[Section 7]{SZhao-pants}), however these can be {\em negative} because of the presence of curvature.

We resolve this question, as in \cite{FOOO-polydiscs}, by counting the Floer solutions $u$ with weight $T^{E_{top}(u)},$ and subsequently multiplying the resulting map $O$ by $T^C$ for a constant $C$ greater than the uniform norm of the (negative contribution of the) curvature over the corresponding compactified universal curve. Then $T^C O$ is defined over $\lambzero.$ 

For example, given Hamiltonians $F, G,$ we pick $C_1 = \cl{E}_{+}(F-G)+ \epsilon,$ and $C_2 = {\cl{E}_{+}(G-F)}+\epsilon.$ Then, considering small perturbations of the resulting equations, we obtain Floer continuation maps $T^{C_1} C(F,G),$ $T^{C_2}C(G,F),$ such that their compositions are homotopic over $\lamzero$ to $T^{C_1 + C_2} \id$ in either direction.

\subsubsection{Local Floer homology and a canonical $\lambzero$-complex}\label{subsubsec: local FH}

In this section we describe the local Floer homology $HF^{\loc}(\phi,x)$ of $\phi$ at an isolated fixed point $x \in \fix(\phi).$ We refer to \cite{Ginzburg-CC, GG-local-gap} as well as \cite{SZhao-pants} for more details on the description of local Floer cohomology. 

In various settings, for example in the symplectically aspherical case, it is known that the local Floer homologies act as "building blocks" for the global Floer homology. For example, the action filtration induces a spectral sequence converging to the Floer homology in an suitable action window. In the monotone case, the approach of the action filtration runs into the difficulty that the same orbit $x$ can contribute to the homology in the action window several times, entering with different cappings, and it is difficult to distinguish the different contributions.

In order to deal with this issue, we develop a version of the above spectral sequence by working with coefficients in $\lambzero_{\tuniv},$ removing the recapping ambiguity, and describing a homotopy-canonical $\lambzero$ complex on the sum of the local homology groups that computes the bar-length spectrum in the barcode of $\phi.$ The latter is known to be finite: see \cite{LSV-conj} and \cite{S-Zoll}.

{\em Local Floer homology:}

Let $\phi \in \Ham(M,\om).$ Given an isolated fixed point $x$ of $\phi,$ there exists an isolating neighborhood $U$ of $x$ (more precisely, of the image of the section $\sigma_x(t) = x(t)$ of $S^1 \times M \to S^1$) for Floer homology. This means in particular, all Floer trajectories of each sufficiently $C^2$ small non-degenerate Hamiltonian perturbation $\phi_1$ of $\phi$ between generators contained in $U$ are themselves contained in $U,$ and the resulting Floer homology as computed inside $U,$ is well-defined and independent of the perturbation. This homology is called the local Floer homology $HF^{\loc}(\phi,x)$ of $\phi$ at $x.$ As suggested by the notation, whenever the local Floer homology is considered as an {\em ungraded} $\bb K$-module it depends on $\phi, x$ and no additional data. We refer to \cite{Pozniak, Floer-MorseWitten, Floer3} for earlier developments in the subject.

We recall the following additional properties of $HF^{\loc}(\phi,x).$ First if $x$ is non-degenerate as a fixed point of $\phi,$ then as ungraded $\bb K$-modules, \[HF^{\loc}(\phi,x) \cong \bb K.\] In fact, there is a canonical $\Z/(2)$-grading on $HF^{\loc}(\phi,x).$ Then in the non-degenerate case, the $\bb K$ factor will be in the component of $CZ(\til{\phi},\ol{x})+n$ mod $2.$ Furthermore, in view of action arguments in \cite{GG-revisited,GG-hyperbolic,McLean-geodesics} or \cite[Section 7]{SZhao-pants}, for two distinct fixed points $x,y \in \fix(\phi),$ there exists a {\em crossing energy} $2\epsilon_0 > 0,$ such that all Floer trajectories, or product structures considered in this paper, with $x,y$ among their asymptotics, carry energy of at least $2\epsilon_0.$ This shall be important for the following arguments.

We call an iteration $\phi^k$ of $\phi$ {\em admissible} at a fixed point $x \in \fix(\phi)$ if $\lambda^k \neq 1$ for all eigenvalues $\lambda\neq 1$ of $D(\phi)_x.$ Observe that all sufficiently large prime iterations $p,$ and their powers $p^n,$ are admissible. We shall require the following result, which is part of \cite[Theorem 1.1]{GG-local-gap}, on the behavior of local Floer homology under iteration. It does not depend on the choice of coefficients.

\begin{thm}\label{thm: GG persistence}
Let $k$ be an admissible iteration of $\phi$ at $x.$ Then the fixed point $x^{(k)}$ of $\phi^k$ is isolated, and $HF^{\loc}(\phi^k,x^{(k)}) \cong HF^{\loc}(\phi,x),$ 
\end{thm}

{\em The canonical $\lambzero$-complex:}

Since this point is somewhat new, we describe it in more detail. Let $\phi_1$ be a sufficiently $C^2$-small non-degenerate Hamiltonian perturbation of $\phi.$ Then for each $x \in \fix(\phi),$ the periodic orbits $\cO(\phi_1,x)$ of $\phi_1$ in the neighborhood $U$ of $\phi$ form the local Floer complex $CF^{\loc}(\phi_1,x)$ of $\phi_1$ at $x.$ Consider the Floer complex $CF(\phi_1,\lambzero)$ of $\phi_1$ with coefficients in $\lambzero.$ Note that up to isomorphism of $\lambzero$-modules it does not depend on the choice of almost complex structures involved. We also mention that as our symplectic manifold is monotone, all elements $l = \sum a_j T^{\lambda_j}$ in $\lambzero$ involved in the differential are in fact {\em finite} in the sense that $a_j = 0$ for all $j$ sufficiently large. Our goal is to construct a homotopically-canonical $\lambzero$-complex $CF(\phi,\lambzero)$ with the following properties: as a $\lambzero$-modules it is given by \[CF(\phi,\lambzero) \cong \bigoplus_{x \in \fix(\phi)} HF^{\loc}(\phi,x) \otimes_{\bK} \lambzero,\] its differential is defined, strict, and is homotopy-canonical over $\lambzero,$ the homology of $CF(\phi,\Lambda) = CF(\phi,\lambzero) \otimes_{\lambzero} \Lambda$ is isomorphic to $HF(\phi_1,\Lambda) \cong QH_*(M,\Lambda),$ and the associated bar-length spectrum \[\beta'_1(\phi,\bK) \leq \ldots \leq \beta'_{K(\phi,\bK)} (\phi,\bK)\] satisfies $\beta'_1(\phi,\bK) > \epsilon_0,$ is $2\delta_0$-close to the part \[\beta_{K'+1}(\phi_1,\bK) \leq \ldots \leq \beta_{K'+K(\phi,\bK)} (\phi_1,\bK)\] of the bar-length spectrum of $\phi_1$ above $\epsilon_0,$ while $\beta_{K'}(\phi_1,\bK) < 2\delta_0 \ll \epsilon_0.$ For $\delta_0 \ll \eps_0$ a small parameter converging to $0$ as $\phi_1$ converges to $\phi$ in the $C^2$-topology on the Hamiltonian. Moreover, the $\beta'_j(\phi,\bK)$ for $1 \leq j \leq K(\phi,\bK)$ have a limit $\beta_j(\phi, \bK)$ as the Hamiltonian perturbation goes to zero in the $C^2$ topology.

Furthermore, as a consequence (see Lemma \ref{lma:endpoints}), for \[N(\phi,\bK) =  \displaystyle\sum_{x \in \fix(\phi)} \dim_{\bK} HF^{\loc}(\phi,x)\] and $B(\bK) = \dim_{\bK} H_*(M,\bK)$ we have \[ N(\phi,\bK) =  2 K(\phi,\bK) + B(\bK).\]

The construction proceeds as follows. We first consider $CF(\phi_1,\lambzero).$ By the key property of the crossing energy, the differential $d_{\phi_1}$ in this Floer complex satisfies \[d_{\phi_1} = d_{\loc,\phi_1} + T^{\epsilon_0} D\] for an operator $D$ defined over $\lambzero,$ and the differential $d_{\loc,\phi_1}$ is the direct sum of differentials in the local Floer complexes $CF^{\loc}(\phi_1,x)$ taken with $\lambzero$ coefficients. Furthermore, as $\phi_1$ is a $C^2$-small perturbation of $\phi,$ by standard action-energy estimates in Floer theory (see for example \cite[Section 7]{SZhao-pants}), all the powers of $T$ involved in \[d_{\loc,\phi_1}: CF(\phi_1) \to CF(\phi_1)\] are at most $\delta_0 \ll \epsilon_0.$ 

In fact, considering a capping $\ol{x}$ of the orbit obtained from $x \in \fix(\phi)$ by a Hamiltonian flow of $H \in \cl{H}$ generating $\phi = \phi^1_H,$ for a sufficiently small perturbation $\phi_1,$ the orbits obtained by the perturbed Hamiltonian flow of $H_1 \in \cl{H}$ from fixed points in $\cO(\phi_1,x)$ will inherit cappings from $\ol{x}.$ Moreover, all their actions will be all $\delta_0/2$-near $\cA_H(\ol{x}).$ They yield the local Floer complex $CF^{\loc}(H_1, \ol{x}),$ where we ignore grading. Finally, as in Section \ref{subsubsec: bar-length three flavors}, the torsion exponents of $CF^{\loc}(\phi_1,x)$ are identified with the bar-length spectrum of $CF^{\loc}(H_1, \ol{x}).$ Furthermore, as in Section \ref{subsubsec: barcode isolated}, each of the latter bar-lengths is given as the difference of the $H_1$-actions of two generators of $CF^{\loc}(H_1, \ol{x}),$ and is hence at most $\delta_0.$  Finally, $\delta_0 \ll \epsilon_0$ can be made arbitrarily small by choosing $\phi_1$ sufficiently close to $\phi.$

Choose a $\lambzero$-basis of $CF(\phi_1)$ compatible with $d_{\loc,\phi_1}:$  \[(\bar{\xi}_1,\ldots,\bar{\xi}_{B},\bar{\eta}_1,\ldots,\bar{\eta}_{K},\bar{\zeta}_1,\ldots,\bar{\zeta}_{K})\]
\[B = N(\phi,\bK) = \sum_{x \in \fix(\phi)}\dim_{\bK} HF^{\loc}(\phi,x),\] \[N' = N'(\phi_1,\bK) = \sum_{x \in \fix(\phi)} \dim_{\bK} CF^{\loc}(\phi_1,x) \otimes_{\lambzero} {\bK},\] 
\[ K' = K'(\phi_1,\bK) = (N' - B)/2,\] where we consider $\bK$ as the quotient of $\lambzero$ by its unique maximal ideal, with \[d_{\loc,\phi_1} (\bar{\xi}_j) = 0,\; d_{\loc, \phi_1} \bar{\zeta}_j = T^{\delta_j} \bar{\eta}_j,\] where $\delta_j = {\beta_j(\phi_1,\bK)} < \delta_0 \ll \epsilon_0.$ 

Consider $X = \lambzero \left< \{ \bar{\xi}_j\} \right>,$ isomorphic to the free part of the homology of the complex $(CF(\phi_1,\lambzero),d_{\loc,\phi_1}),$ $\pi_X: CF(\phi_1,\lambzero) \to X$ the natural projection, and $\iota_X: X \to CF(\phi_1,\lambzero)$ the natural inclusion. Furthermore, let $\Theta: CF(\phi_1, \Lambda) \to CF(\phi_1,\Lambda)$ be the operator \[ \Theta(\bar{\eta}) = T^{-\delta_j} \bar{\zeta}.\] (Note that $\Theta$ is {\em not} defined over $\lambzero$!)

Now, after tensoring with $\Lambda,$ there is a standard homotopy-canonical way, called the homological perturbation lemma \cite{Markl-ideal}, of constructing a differential $d_{\phi}$ on $H(CF(\phi_1,\Lambda),d_{\loc,\phi_1})$ so that 
\[(H(CF(\phi_1,\Lambda),d_{\loc,\phi_1}), d_{\phi}) = H(CF(\phi_1,\Lambda), d_{\phi_1}).\]

It remains to check that $d_{\phi}$ obtained in this way is in fact defined over $\lambzero$ and satisfies the above properties. Let us check well-definedness: indeed, $d_{\phi}$ is given by the formula: \[ d_{\phi} = \pi_X (T^{\epsilon_0}D + T^{2\epsilon_0} D \Theta D + T^{3\epsilon_0} D \Theta D \Theta D + \ldots ) \iota_X.\] Since $\delta_0 \ll \epsilon_0,$ the statement is now evident.

Finally, let us observe that $\pi_X: CF(\phi_1,\lambzero) \to X, \iota_X: X \to CF(\phi_1,\lambzero)$ upgrade to $\Lambda$-homotopy equivalences $\ol{\pi}_X,$ $\ol{\iota}_X$ between $CF(\phi_1,\Lambda)$ and the perturbed complex $(X, d_{\phi}),$ given by 

\[\ol{\pi}_X = \pi_X + \pi_X(T^{\epsilon_0}D + T^{2\epsilon_0} D \Theta D + T^{3\epsilon_0} D \Theta D \Theta D + \ldots) \Theta \]
\[\ol{\iota}_X = \iota_X + \Theta(T^{\epsilon_0}D + T^{2\epsilon_0} D \Theta D + T^{3\epsilon_0} D \Theta D \Theta D + \ldots) \iota_X.\]

These maps are similarly defined over $\lambzero,$ and show that over $\lambzero,$ $H(X, d_{\phi}) \cong H(CF(\phi_1,\lambzero))^{>\epsilon_0},$ the latter denoting the direct summand in $H(CF(\phi_1,\lambzero))$ containing the free part, and the torsion parts of torsion exponent $> \epsilon_0.$ Finally, \[\ol{\pi}_X \circ \ol{\iota}_X = \id,\] and over $\Lambda,$ \[\ol{\iota}_X \circ \ol{\pi}_X - \id = d \ol{\Theta} + \ol{\Theta} d \] where \begin{equation}\label{eq: ol Theta} \ol{\Theta} = \Theta + \Theta (T^{\epsilon_0}D + T^{2\epsilon_0} D \Theta D + T^{3\epsilon_0} D \Theta D \Theta D + \ldots) \Theta.\end{equation} As $\delta_0 \ll \epsilon_0,$ this shows that $ T^{\delta_0/2} \ol{\pi}_X, T^{\delta_0/2} \ol{\iota}_X$ is a $\delta_0$-quasi-equivalence between $(X, d_{\phi})$ and $CF(\phi_1,\lambzero).$ Indeed, $T^{\delta_0} \ol{\Theta}$ is then defined over $\lamzero.$

We now make the perturbation $\phi_1\phi^{-1}$ tend to zero in the $C^2$-topology on the Hamiltonian. Observe that the bar-lengths $\beta_j(X, d_{\phi_1})$ are $2\delta_0$-close (for $\delta_0$ tending to zero) to the bar-lengths of $\phi_1$ greater than $\epsilon_0.$ By Section \ref{subsubsec: barcode isolated}, the latter bar-lengths form convergent sequences, and their limits depend only on $\phi.$ Hence, the bar-lengths of $(X,d_{\phi_1})$ converge to $\beta_j(\phi,\bK),$ $1 \leq j \leq K(\phi,\bK)$ as in Section \ref{subsubsec: barcode isolated}.

\section{Proof of Theorem \ref{thm: bound semisimple}}\label{sec: gen proofs}

In this section we prove Theorem \ref{thm: bound semisimple}. It follows immediately from a combination of the following statements, Propositions \ref{prop: beta leq gamma} and \ref{prop: bounded beta K}, proven in Section \ref{section: proofs}.

\begin{prop}\label{prop: beta leq gamma}
	Let $\bK$ be a field. Let $E=(e_1,\ldots,e_S),$ $e_j \in QH_{2n}(M,\bK)$ for all $1 \leq j \leq S$ be idempotents that split $QH_{ev}(M,\bK),$ as an algebra, into a direct sum of algebras. Define for each $\til{\phi} \in \til{\Ham}(M,\om)$ the splitting-modified spectral norm as \[ \gamma_E(\til{\phi},\bK) = \max_{1 \leq j \leq K} \gamma_{e_j}(\til{\phi},\bK),\] \[\gamma_{e_j}(\til{\phi},\bK) =  c(e_j,\til{\phi},\bK) + c(e_j,\til{\phi}^{-1},\bK).\] Put $\gamma_E(\phi,\bK) = \inf_{pr(\til{\phi}) = \phi} \gamma_E(\til{\phi},\bK).$
	Then the bar-length spectrum of $\phi$ over $\bK$ is coarse Lipschitz in the pseudo-distance induced by the splitting-modified spectral norm over $\bK$. In fact  \[|\beta_j(\phi,\bK) - \beta_j(\psi,\bK)| \leq \gamma_E (\phi \psi^{-1},\bK) + 2n\] for all $j \in \Z_{\geq 0}.$ In fact the same is true for the class of the barcode modulo uniform shifts, in the bottleneck distance. 
\end{prop}

\medskip

The following is a celebrated result of Entov and Polterovich \cite{EntovPolterovichCalabiQM}, slightly modified for our purposes.

\medskip

\begin{prop}\label{prop: bounded beta K}
	Let $QH_{ev}(M,\bK)$ be semisimple, and $E=(e_{1},\ldots,e_{S}),$ where $e_j \in QH_{2n}(M,\bK)$ for all $1 \leq j \leq S$ be the idempotents that split $QH_{ev}(M,\bK),$ as an algebra, into a direct sum of fields. Then there exists a constant $0 \leq D_E(\bK) \leq 6n$ such that for each $\til{\phi} \in \til{\Ham}(M,\om)$ the rational splitting-modified spectral norm $\gamma_E(\til{\phi},\bK)$ is at most $D_E(\bK).$ Therefore for each $\phi \in \Ham(M,\om),$ \[\gamma_E(\phi,\bK) \leq D_E(\bK).\]
\end{prop}

\medskip

\section{Algebraic statements}\label{sec: algebra}

We collect a list of algebraic statements, proven in Section \ref{section: proofs}, that are necessary for the proof of the main theorem.

\begin{prop}\label{prop: bounded beta F_p}
Let $QH_{ev}(M,\Q)$ be semisimple, where $(M,\om)$ is a monotone symplectic manifold. Then for all $p \geq p_0,$ where $p_0$ is large enough, $QH_{ev}(M,\F_p)$ is semisimimple, with idempotents $E_p = (e_{1,p},\ldots,e_{S_p,p})$ splitting it into a direct sum of fields, and there exists a constant $D_{E_p}(\F_p) \leq 6n$ such that for each $\phi \in \Ham(M,\om),$ \[\gamma_{E_p}(\phi, \F_p) \leq D_{E_p}(\F_p).\] 
\end{prop}

\medskip

\begin{rmk}\label{rmk: reduction mod p of idempotents}
Moreover $\{1,\ldots,S_p\} = \sqcup_{1 \leq j \leq S_0} I_j,$ where each $I_j$ is a non-empty set of consecutive integers, such that for each $1 \leq j \leq S_0,$ we have that $\sum_{l \in I_j} e_{j,p}$ equals the reduction of $e_{j,0}$ modulo $p$ (which exists for all $p$ sufficiently large).
\end{rmk}

\medskip
\begin{lma}\label{lemma: beta F_p indep of p}
Let $(M,\omega)$ be a monotone symplectic manifold. Let $\phi$ be a Hamiltonian diffeomorphism with isolated fixed points. Then for $p \geq p_2(\phi),$ the bar-length spectrum $0 < \beta_1(\phi,\F_p) \leq \ldots \leq \beta_{K(\phi,\F_p)}(\phi,\F_p)$ of $\phi$ over $\F_p$ coincides with the bar-length spectrum $0 < \beta_1(\phi,\Q) \leq \ldots \leq \beta_{K(\phi,\Q)}(\phi,\Q)$ of $\phi$ over $\Q.$ In particular $\beta(\phi,\F_p) = \beta(\phi,\Q)$ for $p \geq p_2(\phi).$
\end{lma}

\medskip


\begin{prop}\label{prop: algebraic deformation inequality}
Let $(C,d_0)$ be a strict filtered complex over $\lambzero_{\bK},$ for ground field $\bK.$ Let $0 < \beta_1\leq \beta_2 \leq \ldots \leq \beta_K = \beta$ be the bar-length spectrum of $(C,d_0).$ Let $\K = \bK[u^{-1},u]]$ be a completed transcendental extension of $\bK.$ Thus $\K$ is the field of fractions of $\cL = \bK[[u]].$ Consider a (not necessarily strict) filtered complex $(C\otimes \K,d)$ over $\lambzero_{\K},$ where $C \otimes \K := C \otimes_{\lambzero_{\bK}} \lambzero_{\K}$ under the natural inclusion $\lambzero_{\bK} \to \lambzero_{\K},$ whose differential $d: C \otimes \K \to C \otimes \K$ is of the form \[d = d_0 +  u D,\] for a map $D \in Hom_{\lambzero_{\K}} (C \otimes \K, C \otimes \K),$ given by $D = \Phi(\overline{D}),$ with $\overline{D} \in Hom_{\lambzero_{\cL}} (C \otimes \cL, C \otimes \cL),$ where $\Phi:Hom_{\lambzero_{\cL}} (C \otimes \cL, C \otimes \cL) \to Hom_{\lambzero_{\K}} (C \otimes \K, C \otimes \K)$ is the natural map induced by the inclusion $\cL \to \K$ (in other words, all coefficients of $D$ have non-negative $u$-degree). Assume that $H(C \otimes \K,d) \otimes {\Lambda_{\K}} = H(C,d_0) \otimes \Lambda_{\K}.$ Denote by $0 \leq \check{\beta}_1 \leq \ldots \leq \check{\beta}_K = \check{\beta}$ the verbose bar-length spectrum of $(C \otimes \K, d).$ (Here we allow zero torsion exponents, as the complex was non-strict.)  Then for each $j,$ \[ \check{\beta_1} +\ldots + \check{\beta_j} \leq {\beta_1} +\ldots +{\beta_j}.\] In other words, the spectrum $\{ \wh{\beta}_j \}$ is majorized by the spectrum $\{ \beta_j \}.$ In particular, for $j=K,$ we obtain \[\check{\beta}_{\tot} \leq \beta_{\tot}.\]  
\end{prop}

\begin{lma}\label{prop: alg cone}
	Let $(C,d)$ be a strict free complex of finite rank over $\lambzero,$ and $S:C \to C$ be a $\lambzero$-chain map, that induces the zero map on $H(C,d) \otimes \Lambda.$ Then setting $Cone(S) = Cone(S:C \to C)$ for the cone of $S,$ we have \[\beta_{tot}(Cone(S)) \leq 2 \cdot \beta_{\tot}(C).\]
\end{lma}


\medskip
{\begin{lma}\label{lma:endpoints}
Let $(C,d_0)$ be a strict filtered complex over $\lambzero_{\bK}.$ Let \[N(C) = \dim_{\Lambda_{\bK}} C \otimes \Lambda_{\bK}, \;\; \text{and} \;\;\; B(C) = \dim_{\Lambda_{\bK}} H(C \otimes \Lambda_{\bK},d_0 \otimes 1).\] Then the number $K(C)$ of finite bars in $\cB(C)$ satisfies \[K(C) = \frac{N(C) - B(C)}{2}.\] Hence the number $K(\phi,\bK)$ of finite bars of $\cB(\til{\phi}),$ where $\til{\phi} \in \til{\Ham}(M,\om)$ has a non-degenerate time-one map $\phi = pr(\til{\phi}),$ the number $N(\phi)$ of contractible fixed points of $\phi,$ and the number $B(\phi,\bK) = B(\bK) = dim_{\Lambda_{\bK}} QH_*(M,\bK) = \dim_{\bK} H_*(M,\bK)$ of infinite bars in $\cB(\til{\phi})$ are related by \[K(\phi,\bK) = \frac{N(\phi) - B(\bK)}{2}.\] Similarly, in the degenerate, but isolated case, with $N(\phi,\bK)$ from \eqref{eq: homological count N}. In particular, in the non-degenerate case, the condition $N(\phi) = \# \fix(\phi) > \dim_{\bK} H^*(M,\bK),$ and in the degenerate case $N(\phi,\bK) > \dim_{\bK} H^*(M,\bK),$ implies that $\beta(\phi) > 0.$
\end{lma}}







\section{The $\Z/(p)$-equivariant Tate homology and natural isomorphisms}\label{section: Tate outline}

We overview the steps from \cite{Seidel-pants} and \cite{SZhao-pants} required to construct the $\Z/(p)$-equivariant Tate homology of $\phi^p,$ with $\bF_p$-coefficients, and its isomorphism with the Floer homology of $\phi,$ with a suitable change of coefficients. We outline the main new points necessary in the proof, proving them in detail when necessary, and referring to \cite{Seidel-pants,SZhao-pants} for details when they are not new. The novelty of this section consists in adapting the setting and arguments to the monotone case, to coefficients in $\lambzero,$ and to the local-to-global setting as in Section \ref{subsubsec: local FH}. This turns out to capture precisely the right information from the filtered Floer complex.


The main specific goal of this section is to prove the following result. 

\begin{thm}\label{thm: beta tot}
Let $\phi \in \Ham(M,\om)$ be a Hamiltonian diffeomorphism of a closed monotone symplectic manifold $(M,\om).$ Suppose that $\fix(\phi^p)$ is finite. Then \[ p\cdot\beta_{\tot}(\phi,\F_p) \leq \beta_{\tot}(\phi^p,\F_p).\]
\end{thm}

\subsection{The $\Z/(p)$-equivariant Tate homology over $\lambzero$}


For this section, due to convergence and completeness issues, it is important to take coefficients in $\lambzero_{\K},$ for the well-chosen field $\K =\F_p[u^{-1},u]].$  This is a certain completion of $\lambzero_{\F_p} \otimes_{\F_p} \K.$ It will also be convenient to set $\rp = \F_p[[u]] \extheta,$ and $\hrp = \F_p [u^{-1},u]] \extheta = \K \extheta,$ for $\extheta$ the exterior algebra on the formal variable $\theta$ of degree $1.$

First of all, supposing that $\phi$ is non-degenerate, one can compute the $\Z/(p)$-equivariant Tate homology \[H_{Tate}(\Z/(p),CF(\phi,\lambzero_{\F_p})^{\otimes p}),\] by the $\Z/(p)$-equivariant Tate complex \[C_{Tate}(\Z/(p),CF(\phi,\lambzero_{\F_p})^{\otimes p}),\] which, as a module over $\lambzero_{\K}$ is \[CF(\phi,\lambzero_{\F_p})^{\otimes p} \otimes_{\lambzero_{\F_p}} \lambzero_{\K} \left < \theta \right >,\] where $\left < \theta \right >$ denotes the free exterior algebra on an element $\theta$ of degree $1$ (hence isomorphic to $\F_p^2$ as a vector space over $\F_p$). Let $d^{(p)}$ be the differential on $CF(\phi,\lambzero_{\F_p})^{\otimes p}$ naturally arising from the Floer differential $d_{\phi}$ on $CF(\phi,\lambzero_{\F_p}).$ Note that $d_{\phi}^2 = 0$ for a generic almost complex structure by monotonicity of the symplectic manifold: configurations with bubbling do not appear in the moduli spaces we consider for index-transversality reasons. Then the Tate differential on $C_{Tate}(\Z/(p),CF(\phi,\lambzero_{\F_p})^{\otimes p})$ is a $\lambzero_{\K}$-linear extension of \[d_{Tate} x = d^{(p)} x + \theta (1-\tau) x,\] \[d_{Tate}(\theta x) = \theta d^{(p)} x + u (1 + \tau + \ldots + \tau^{p-1}) x,\] for $x = x_1 \otimes \ldots \otimes x_p,$ $x_j \in \fix(\phi), \; 1\leq j \leq N(\phi),$ an element of the standard basis of $CF(\phi,\lambzero_{\F_p})^{\otimes p} \otimes_{\lambzero_{\F_p}} {\lambzero_{\K}}$ over $\lambzero_{\K}.$ Here $\tau$ is induced from its action \[\tau( x_0 \otimes \ldots \otimes x_{p-1} ) = (-1)^{|x_{p-1}|(|x_0|+\cdots + |x_{p-2}|)}x_{p-1} \otimes x_0 \otimes \cdots \otimes x_{p-2}\] on the elements of the standard basis. Note that $\tau^p = 1,$ and hence it generates a $\Z/(p)$-action on the complex $(CF(\phi,\lambzero_{\F_p})^{\otimes p},d^{(p)}).$ 

Supposing that $\phi^p$ is non-degenerate, we can upgrade the complex \[(CF(\phi^p,\lambzero_{\K}),d_0)\] to the Tate complex \[(CF_{Tate,\Z/(p)}(\phi^p,\lambzero_{\F_p}),d_{Tate}).\]

This is carried out by considering Floer equations parametrized by the negative Morse flowlines of a natural $\zp$-invariant Morse function $\til{f}$ and Riemannian metric on $S^{\infty}$ (thought of as a Morse function on $L_p^{\infty} = S^{\infty}/(\Z/(p))$) obtained from the Morse function $f = \sum_{j=0}^{\infty} j |z_j| ^2 / \sum_{j=0}^{\infty} |z_j| ^2 $ on $\C P^{\infty}.$ It turns out to be sufficient to consider the parametrized flowlines $\cP^{i,m}_{\alpha},$ for $m \in \zp,$ $\alpha \in \{0,1\},$ and $i  > \alpha$ is an integer. This is the space of flowlines from the $m$-th critical point $Z^m_i$ of index $i$ of $\til{f}$ to $Z^0_{\alpha}.$ By an index-transversality argument that prevents bubbling, after introducing the requisite perturbations, the Tate differential is well-defined and satisfies $d_{Tate}^2 = 0.$ It has the form:
\[d_{Tate} = d_0 + d_1 + d_2 + \ldots,\] where $d_0(x \otimes 1),$ $d_0(x \otimes \theta)$ are given by the Floer differential of $\phi^p,$ while \[d_1 (x \otimes 1) = d_{1}^{v} x \otimes \theta = (C_p - R_{1/p}) x \otimes \theta,\] \[d_1 (x \otimes \theta) = u \cdot d_{1}^{h} x \otimes 1 = u \cdot (C_p + R_p + R_{2/p} + \ldots + R_{(p-1)/p}) x \otimes 1,\] where $R_{j/p}$ is the loop rotation operator by $j/p$ of the full angle (see \cite{Seidel-pants, PolShe}). It is given by a Floer continuation map $C(J,(\phi^{j})^\ast J)$ composed with the map $x \mapsto \phi(x),$ all taken with a suitable sign (see \cite{Tonkonog-commuting, CineliGinzburg}). We observe that on homology $[R_{j/p}] = [R_{1/p}]^j,$ and furthermore $[R_{1/p}]^p = \id,$ whence it generates a $\zp$-action. We set \[S_p = C_p - R_{1/p}.\] More generally \[d_{2j+1}(x \otimes 1) = u^j d_{2j+1}^{v}(x) \otimes \theta,\] \[d_{2j+1}(x \otimes \theta) = u^{j+1}  \cdot d_{2j+1}(x) \otimes 1,\] \[d_{2j}(x \otimes 1) = u^j \cdot d_{2j}^{h} (x) \otimes 1,\] \[d_{2j}(x \otimes \theta) = u^{j} \cdot d_{2j}^{h}(x) \otimes \theta.\]

The homology \[ HF_{Tate,\Z/(p)}(\phi^p,\lambzero_{\F_p})\] of this complex will be called the $\Z/(p)$-equivariant Tate homology of $\phi^p.$ Note that it is a module over $\lambzero_{\K}.$

\subsection{quasi-Frobenius isomorphism} 

This section is dedicated to the following $\lambzero$-version of a curious algebraic statement from \cite{Kaledin} (see also \cite{SZhao-pants}). 

\begin{lma}\label{lma: quasi Frob}
The (non-linear!) map \[CF(\phi,\lambzero_{\F_p}) \to CF(\phi,\lambzero_{\F_p})^{\otimes p},\; x \to x \otimes \ldots \otimes x,\]  induces an isomorphism of $\lambzero_{\K}$-modules \[r_p^* H_{Tate}(\Z/(p),CF(\phi,\lambzero_{\F_p})) \to H_{Tate}(\Z/(p),CF(\phi,\lambzero_{\F_p})^{\otimes p}),\] where on the left hand side, $r_p: \lambzero_{\K} \to \lambzero_{\K}$ is the homomorphism defined by $T \mapsto T^{\frac{1}{p}},$ and $CF(\phi,\lambzero_{\F_p})$ is considered with the trivial $\Z/(p)$-action.
\end{lma}

Note that since the $\zp$-action on the left is trivial, $H_{Tate}(\Z/(p),CF(\phi,\lambzero_{\F_p})) \cong HF(CF(\phi,\lambzero_{\F_p})) \otimes \K\extheta.$ Therefore we get the isomorphism \[r_p^* HF(\phi,\lambzero_{\K}) \otimes \extheta \xrightarrow{qF} H_{Tate}(\Z/(p),CF(\phi,\lambzero_{\F_p})^{\otimes p}).\] 

We present the proof of this $\lambzero$ modification, which is new, here. 

\begin{proof}[Proof of Lemma \ref{lma: quasi Frob}]
This is a general algebraic statement. Let $(V,d)$ be a strict free $\lambzero$-complex of finite rank. Then \[ r_p^* H(V) \otimes \lambzero_{\F_p} \cong H_{Tate} (\zp, V^{\otimes p}).\] By the same argument as for the non-$\lambzero$ version from \cite{SZhao-pants}, the following property is immediate: for $V_1, V_2$ strict free $\lambzero$-complexes of finite rank, and $V = V_1 \oplus V_2,$ \[ H_{Tate} (\zp, V^{\otimes p}) \cong H_{Tate} (\zp, V_1^{\otimes p}) \oplus H_{Tate} (\zp, V_2^{\otimes p}).\] By choosing an adapted $\lambzero$-basis on the original $V,$ that splits $V$ into a finite sum of $\lambzero$-complexes of the form $V_1 = (\lambzero \brat{x} , 0)$ and $V_2 = (\lambzero \brat{z} \oplus \lambzero \brat{y}, d(z) = T^{\beta} z),$ it is therefore enough to prove that in the second case \[H_{Tate} (\zp, V_2^{\otimes p}) \cong (\lambzero/T^{p\beta} \lambzero)  \otimes_{\lambzero} \lambzero_{\K}\extheta.\] We proceed as follows. First, we note that as over $\Lambda,$ $V_2 \otimes_{\lambzero} \Lambda$ is acyclic, so is \begin{equation}\label{eq: scaling Tate} H_{Tate} (\zp, (V_2 \otimes_{\lambzero} \Lambda) ^{\otimes p}).\end{equation} This implies that $H_{Tate} (\zp, V_2^{\otimes p})$ is a torsion module.

Now, consider the base $\{x = x_0 \otimes \ldots \otimes x_{p-1} \}$ of $V_2^{\otimes p},$ where $x_j \in \{y, z\}$ for all $j.$ Set the following auxiliary degrees: $|z| = 1, |y| = 0.$ This induces a {\em homological grading} on $V_2^{\otimes p},$ with unique base elements $z \otimes \ldots \otimes z$ of maximal degree $p,$ and $y \otimes \ldots \otimes y$ of minimal degree $0.$ 

We observe that the Tate differential preserves the filtration induced by this grading, where $\cF^d$ is generated by all elements of degree at most $d,$ and on homogeneous elements $x$ of positive degree, the non-Archimedean valuations satisfy \begin{equation} \label{eq: alg Tate val} \nu (d_{Tate}(x)) = \beta + \nu (x).\end{equation}

Considering the associated spectral sequence of $\lambzero_{\K}$-modules, we see by the non-$\lambzero$ argument that the $E_1$ page is given by \[\lambzero_{\K} \brat{z \otimes \ldots \otimes z} \otimes \extheta \oplus \lambzero_{\K} \brat{y \otimes \ldots \otimes y} \otimes \extheta.\] In particular, the only possibly non-trivial differential is on the $E_p$ page. By the acyclicity over $\Lambda,$ it must indeed be non-trivial, and by \eqref{eq: alg Tate val}, as well as the action of $\extheta,$ it must be given by a multiplication by an non-zero $\F_p((u))$-multiple of an upper triangular unipotent matrix $P \in \mrm{Mat}(2,\F_p)$ times $T^{p\beta}.$ This finishes the proof. 

Another, more precise argument for \eqref{eq: scaling Tate}, is as follows. Rewrite the Tate differential $d_{Tate}$ as \begin{equation} \label{eq: Tate split for perturbation} d_{Tate} = d_{Tate,lin} + d^{(p)} \otimes 1.\end{equation} Observe that $d^{(p)} \otimes 1 = T^{\beta} D,$ for an operator $D$ of valuation $0$ defined over $\lambzero.$ Furthermore, $d_{Tate,lin}$ is the coefficient-extension of the Tate differential $\overline{d}_{Tate}$ on the finite-dimensional vector space $\bar{V} = V \otimes_{\lambzero} \F_p.$ Therefore the homology of the Tate complex, but with $d_{Tate,lin}$ is given by the coefficient-extension of \[ H_{Tate} (\zp, \bar{V}^{\otimes p})  = \brat{z \otimes \ldots \otimes z} \otimes \F_p((u))\extheta \oplus \brat{y \otimes \ldots \otimes y} \otimes \F_p((u)) \extheta.\] That is: \[H(C_{Tate}(\zp, V_2^{\otimes p}), d_{Tate,lin}) =   \lambzero_{\K} \brat{z \otimes \ldots \otimes z} \otimes \extheta \oplus \lambzero_{\K} \brat{y \otimes \ldots \otimes y} \otimes \extheta.\]

Considering the obvious inclusion $\ol{\iota}: H_{Tate}(\zp, \ol{V}^{\otimes p}) \to C_{Tate}(\zp, \ol{V}^{\otimes p}),$ and projection $C_{Tate}(\zp, \ol{V}^{\otimes p}) \to H_{Tate}(\zp, \ol{V}^{\otimes p})$ given by separating the degree $p,$ and $0$ components, there exists a homotopy operator $\ol{\Theta}$ defined over $\F_p((u)),$ such that \[ \ol{\pi} \circ \ol{\iota} - \id = \ol{d}_{Tate} \ol{\Theta} + \ol{\Theta}\, \ol{d}_{Tate}.\] Now, while this is not strictly necessary, it helps to observe first that we may restrict ourselves to the field of Laurent polynomials $\F_p[u^{-1},u],$ since all constructions so far were finite. Furthermore, setting $u$ to be of degree $-2,$ and $\theta$ of degree $-1,$ we obtain that $\ol{d}_{Tate}$ is of degree $-1$ and $\ol{\Theta}$ is an operator of degree $1.$ 

Now, we extend coefficients to $\lambzero_{\K}$ everywhere, and obtain maps $\iota, \pi, d_{Tate,lin}, \Theta$ satisfying the above relations. Having done this, we may apply homological perturbation \cite{Markl-ideal} to obtain a differential $\til{d}$ on \[W \otimes \extheta  = \lambzero_{\K} \brat{z \otimes \ldots \otimes z} \otimes \extheta \oplus \lambzero_{\K} \brat{y \otimes \ldots \otimes y} \otimes \extheta\] that computes $H_{Tate} (\zp, V_2^{\tens p}),$ and is acyclic over $\Lambda_{\K}.$ This differential is given by the following sum: \[\til{d} = \pi( T^{\beta} D + T^{2\beta} D \Theta D + \ldots ) \iota.\] By consideration of the homological degree, the unique non zero term in this sum is \[\til{d} =  T^{p \beta} D \Theta \ldots \Theta D,\] where $D$ appears $p$ times, and $\Theta$ appears $(p-1)$ times. It must be invertible by acyclicity over $\Lambda,$ and by consideration of the $u,\theta$-degree it preserves the splitting $W \otimes \extheta = W \otimes 1 \oplus W \otimes \theta.$ Furthermore, by the $\theta$-action, we obtain that $\til{d}$ is the multiplication by a non-zero multiple of $T^{p\beta} u^{(p-1)/2}.$

\end{proof}

\begin{rmk}\label{rmk: beta base changed}
Note that the verbose bar-length spectrum $0<\beta_{p,1} \leq \ldots \leq \beta_{p,2K}$ of $r_p^* HF(\phi,\lambzero_{\K}) \otimes \extheta$ satisfies \[\beta_{p,2j-1} = \beta_{p,2j} = p \cdot \beta_{j}(\phi,\F_p),\] for all $1 \leq j \leq K = K(\phi,\F_p).$ Here we used the easy fact that  $\beta_{j}(\phi,\K)= \beta_{j}(\phi,\F_p)$ for all $j.$
\end{rmk}

\subsection{The $\Z/(p)$-equivariant product-isomorphism}

The claim of this section, is that, similarly to Seidel's construction \cite{Seidel-pants} in the case $p=2,$ and that of Shelukhin-Zhao \cite{SZhao-pants} for $p>2$ in the exact or symplectically aspherical cases, there is a natural chain map \[\cP: C_{Tate}(\Z/(p),CF(\phi,\lambzero_{\F_p})^{\otimes p}) \to CF_{Tate,\Z/(p)}(\phi^p,\lambzero_{\F_p}),\] that induces an isomorphism \[\cP: H_{Tate}(\Z/(p),CF(\phi,\lambzero_{\F_p})^{\otimes p}) \to HF_{Tate,\Z/(p)}(\phi^p,\lambzero_{\F_p})\] on homology. The construction is given by counting families of Floer equations over a $p$-branched cover $S$ of the cylinder $Z = \R \times S^1,$ at a point $0 \in Z,$ with $\zp$-symmetric Hamiltonian Floer data, restricting to $H$ on each of the input ends, and Floer datum $\Hk$ over the output end, such that the curvature terms of the perturbation datum vanish. One way of choosing this data is by considering suitable cylindrical strips \cite{KS-bounds}.

The almost complex structures, and additional perturbations, are parametrized by Morse flowlines of the $\zp$-invariant Morse function $\til{f}$ on $S^{\infty}$ as before. However, because of the necessary Hamiltonian perturbation terms required to achieve transversality, the curvature terms are non-zero in general. In the monotone case, by an action-index argument of Seidel \cite[Section 7.1]{Seidel-pants}, there is a constant $i_1(\phi)$ depending on $\phi,$ such that all terms in the equivariant differential, and in the equivariant product operation, come only from the moduli of Morse trajectories of $\til{f}$ in $S^{i_1(\phi)}.$ This makes the choice of perturbations lie in a compact family (after compactification, and choosing the perturbations consistently with breaking of Morse flowlines), and hence that for each $\delta_1 > 0,$ one may make the uniform norm of the curvature be strictly smaller that $\delta_1.$ All constructions now go through, as by index-transversality reasons, bubbling of spheres does not enter into the relevant counts. 

\subsubsection{Local-to-global argument: non-degenerate case}\label{subsubsec: local to global non-deg}

Once the map is contstructed, to show that it is an isomorphism one runs the following analogue of the argument from \cite{Seidel-pants, SZhao-pants} that used the spectral sequence with respect to the action filtration. We observe that by \cite[Proposition 9]{SZhao-pants} on the chain level $\cP$ is given by $\cP^{\loc} + T^{\epsilon_0} P,$ for a map $P$ defined over $\lambzero,$ and $\cP^{\loc} = \bigoplus_{x \in \fix(\phi)} \cP^{\loc}_x$ is the sum of the local product operations: \[\cP^{\loc}_x: C_{Tate}(\zp, CF^{\loc}(\phi,x)^{\otimes p}) \to CF^{\loc}_{Tate,\zp}(\phi^p,x^{(p)})\] composed with the projection to the submodule generated by elements of the form $x \otimes \ldots \otimes x.$  See more about this in the following section. By \cite{Seidel-pants,SZhao-pants} this local operator $\cP^{\loc}_x$ is invertible.  Writing the differential on $C_{Tate}(\zp,CF(\phi)^{\otimes p})$ as $d_{Tate}^{\otimes p} = d_{Tate, lin}^{\otimes p} + d_{\phi}^{\otimes p}$ and applying the homological perturbation lemma, with respect to it, we obtain a complex on \[\oplus_{x \in \fix(\phi)} H_{Tate}(\zp, CF^{\loc}(\phi,x)^{\otimes p}) = \oplus_{x \in \fix(\phi)} \hrp\] $\lambzero$-quasi-isomorphic to the orginal one (observe that convergence holds since $d_{\phi}^{\otimes p} = T^{\eps_0} D,$ for a certain $\lambzero$-morphism $D$). Similarly, writing the differential on $CF_{Tate,\zp}(\phi^p)$ as $d_{Tate} = d_{Tate,loc} + T^{\eps_0} E,$ for a $\lambzero$-morphism $E,$ we obtain a complex on \[\oplus_{x \in \fix(\phi)} HF^{\loc}_{Tate}(\phi^{p},x^{(p)}) = \oplus_{x \in \fix(\phi)} \hrp\] $\lambzero$-quasi-isomorphic to the original one. Furthermore, precomposing and post-composing $\cP$ with the respective quasi-isomorphisms, we obtain the $\lambzero$-morphism \[ \ol{\cP}:  \oplus_{x \in \fix(\phi)} \hrp \to \oplus_{x \in \fix(\phi)} \hrp \] of the form  \[\ol{\cP} = \cP^{\loc} + T^{\eps_0}\ol{P},\] where $\cP^{\loc}$ is canonically identified with $\cP^{\loc}$ from before. Now, however, $\cP^{\loc}$ is a $\lambzero$-invertible map. Hence \[ \ol{\cP} =   \cP^{\loc}(\id + T^{\epsilon_0}(\cP^{\loc})^{-1} \ol{P}) = (\id + T^{\epsilon_0}\ol{P} (\cP^{\loc})^{-1}) \cP^{\loc}\] is a chain map that is invertible over $\lambzero,$ which is therefore a chain-isomorphism. Therefore $\cl{P}$ is a quasi-isomorphism and hence the bar-length spectra of the $\lambzero$-modules $H_{Tate}(\Z/(p),CF(\phi,\lambzero_{\F_p})^{\otimes p})$ and $HF_{Tate,\Z/(p)}(\phi^p,\lambzero_{\F_p})$ coincide.

Combining this with Lemma \ref{lma: quasi Frob}, and Remark \ref{rmk: beta base changed}, gives us the following result. 

\begin{cor}\label{cor: qF and SP on beta}
Let $\F_p = \F_p,$ $\K = \F_p[u^{-1},u]].$ Let $0 < \beta_1(\phi,\F_p) \leq \ldots \leq \beta_{K(\phi,\F_p)}(\phi,\F_p)$ with $K(\phi,\F_p) = \frac{N(\phi) - B(\F_p)}{2}$ be the bar-length spectrum of $\phi.$ Then the verbose bar-length spectrum $0 \leq \check{\beta}_1(\phi^p) \leq \ldots \leq \check{\beta}_{2K(\phi^p,\F_p)}(\phi^p)$ of the $\Z/(p)$-equivariant Tate homology of $\phi^p$ satisfies: \[\check{\beta}_1(\phi^p) = \ldots = \check{\beta}_{2K(\phi^p,\F_p)-2K(\phi,\F_p)}(\phi^p) = 0,\] while, denoting $\kappa(\phi^p,\F_p) = 2K(\phi^p,\F_p)-2K(\phi,\F_p),$ \[\check{\beta}_{\kappa(\phi^p,\F_p) + 2j -1}(\phi^p) = \check{\beta}_{\kappa(\phi^p,\F_p) + 2j}(\phi^p) = p \cdot \beta_j(\phi,\F_p)\] for all $1 \leq j \leq K(\phi,\F_p).$
\end{cor}

Now, Corollary \ref{cor: qF and SP on beta}, together with the fact that the dimension over $\Lambda_{\K}$ of the homology $HF_{Tate,\Z/(p)}(\phi^p,\lambzero_{\F_p}) \otimes \Lambda_{\K}$ coincides with that of $HF(\phi,\Lambda_{\F_p})$ over $\Lambda_{\F_p},$ that is $\dim_{\F_p} H_*(M,\F_p).$ In particular, by Proposition \ref{prop: algebraic deformation inequality} and Lemma \ref{prop: alg cone}, we directly obtain Theorem \ref{thm: beta tot} for non-degenerate $\phi^p.$ 

The rest of this section is dedicated to extending this argument to the general isolated fixed point case, which presents quite a few technical challenges.


\subsection{Local equivariant Floer cohomology}\label{subsubsec: Local equiv}

The situation described in Section \ref{subsubsec: local FH} readily extends to the case of equivariant Floer cohomology. Indeed, supposing that all fixed points $\fix(\phi^p)$ of $\phi^p$ are isolated, for each capped orbit $\ol{x},$ by \cite{SalamonZehnder}, there is an uppper and a lower bound depending only on $\phi$ and $\dim M$ on the difference between the mean-index of $\ol{x}$ and the possible indices of the corresponding capped orbits $\til{x}$ of a sufficiently $C^2$-small non-degenerate Hamiltonian perturbation of $\phi^p$ in the isolating neighborhood of $x.$ We choose this perturbation to be of the form $\phi_1^p,$ where $\phi_1$ is a sufficiently $C^2$-small non-degenerate perturbation of $\phi.$ Furthermore, the actions of these orbits are very close to the action of $\ol{x}.$ Hence there is an upper and a lower bound, depending only on $\phi$ and $\dim M$ on the reduced actions $\til{\cl A}(x) = \cl{A}_H(\til{x}) - CZ(H,\til{x})$ of all $\til{x}$ for all $x \in \fix(\phi^p).$ In particular, the terms $d^i$ of the equivariant differential vanish for all $i > i_0(\phi),$ independently of the choice of perturbation data, by an argument of Seidel \cite[Section 7.1]{Seidel-pants}. Therefore, it is sufficient to control the perturbation data entering into the definition of equivariant Floer homology in a compact family corresponding to $w \in S^{i_0(\phi)}.$ Hence, \cite[Proposition 9]{SZhao-pants} ensures that the perturbation data can be chosen in such a way that the trajectories of the equivariant differential between generators inside a sufficiently small isolating neighborhood $U_x$ of $x \in \fix(\phi^p)$ stay inside $U.$ Furthermore, the same is true for neighborhoods $U_{\phi^{\sigma}x}$ of $\phi^{\sigma} x$ for $\sigma \in \zp.$ Therefore, by definition of the Tate differential, gluing and compactness, the fixed points of $\phi_1^p$ in $U = \bigcup_{\sigma \in \zp} U_{\phi^{\sigma}x}$ form a complex, and the homology of this complex is independent of the Hamiltonian perturbation $\phi_1^p$ of $\phi^p.$ We call this homology the local Tate Floer cohomology ${HF}^{\loc}_{Tate,\zp}(\phi^p,\zp\, x)$ of the orbit $\zp \, x.$ If $x$ is an iterated fixed point, that is $\phi(x) = x,$ or equivalently $\zp\, x = \{x\},$ then only one isolating neighborhood $U$ of $x$ with respect to $\phi^p$ is necessary, and we abbreviate by $HF^{\loc}_{Tate, \zp}(\phi^p,x^{(p)}).$ If $x$ is not an iterated fixed point, then ${HF}^{\loc}_{Tate,\zp}(\phi^p,\zp\, x) = 0.$ Finally, all equivariant Floer differentials connecting fixed points of $\phi_1^p$ in different isolating neighborhoods of orbits, have energy and topological energy at least $\epsilon_0 > 0.$

The local equivariant Floer cohomology enjoys properties similar to those of usual local Floer homology. For instance if $x$ is non-degenerate as a fixed point of $\phi^p,$ then if $x$ is iterated, we have \[HF^{\loc}_{Tate, \zp}(\phi^p,x^{(p)}) \cong H_{Tate}(\zp,\F_p) = \hrp\] as $\hrp$-modules. 

\subsubsection{Local Tate homology and canonical $\lamzero$-complexes}\label{subsubsec: local Tate canonical}

For $\phi_1$ a sufficiently $C^2$-small Hamiltonian perturbation of $\phi,$ and the perturbation terms in the Tate differential being sufficiently small, the Tate differential is in fact defined over $\lambzero,$ and is of the form \begin{equation}\label{eq: Floer Tate split for perturbation} d_{Tate} = d_{Tate,\loc} + T^{\eps_0} E_{Tate},\end{equation} where $d_{Tate,\loc}$ is the $\lambzero$ version of the local Tate differential, and both $d_{Tate,\loc},$ $E_{Tate}$ are defined over $\lambzero.$ This forms two $\lambzero$-complexes \[CF_{Tate,\zp}^{\loc}(\phi^p_1,\lambzero_{\F_p}),\;\; CF_{Tate,\zp}(\phi^p_1,\lambzero_{\F_p}).\] The dimension over $\Lambda_{\K}$ of $H(CF_{Tate,\zp}^{\loc}(\phi_1^p,\lambzero_{\F_p})) \otimes \Lambda_{\K}$ is the sum \[\sum_{O} \dim_{\K} HF^{\loc}_{Tate,\zp}(\phi^p,O) = \sum_{x \in \fix(\phi)} \dim_{\K} HF^{\loc}_{Tate,\zp}(\phi^p,x^{(p)}),\] where $O$ runs over the $\zp$-orbits $O$ in the set $\fix(\phi^p).$ 

We remark that here we took the perturbation terms in the Tate differential sufficiently small depending on $\phi_1^p$ to make $d_{Tate,\loc}$ and consequently $d_{Tate}$ be defined over $\lambzero$ directly. However, we can alternatively choose perturbations of curvature bounded above by $\delta' \ll \eps_0$ independently of $\phi_1^p,$ and define $d_{Tate}$ following Secion \ref{subsubsec: operations and lambzero}, multiplying by $T^{\delta'}.$ This will result with a $\lambzero$-complex with the same properties of interest, except that all $\delta$-quasi-equivalences become $(\delta'+\delta)$-quasi-equivalences.

Now, as in Section \ref{subsubsec: local FH}, for $\phi_1$ being a sufficiently $C^2$-small Hamiltonian perturbation of $\phi,$ $CF_{Tate, \zp}(\phi_1^p, \lamzero_{\F_p})$ is  $\delta'_0$-quasi-equivalent, where $\delta'_0 = O(\delta_0) \ll \epsilon_0,$ to a complex on \[CF_{Tate, \zp}(\phi^p, \lamzero_{\F_p}) = \bigoplus HF^{\loc}_{Tate, \zp}(\phi^p,x^{(p)}) \otimes_{\K} \lamzero_{\K}\] with differential $d_{Tate,\, \phi^p}$ obtained by homological perturbation from the differential $d_{Tate,\, \phi_1^p}.$ Note that the latter complex is of rank independent of $\phi_1.$ In particular its concise bar-length spectrum has length $K_{Tate}(\phi^p)$ independent of $\phi_1.$ Furthermore, by the properties of the crossing energy, or in other words, the decomposition \eqref{eq: Floer Tate split for perturbation} it is rather quick to show that this complex is strict. 






However, for our application to the proof of Theorem \ref{thm: beta tot}, we require another, more refined, $\lamzero$-complex that is also homotopy-canonically associated to the situation. We shall construct it by repeated application of homological perturbation. Recall that the first two components $d_0 + d_1,$ of the Tate differential yield the differential $d_{Cone}$ on \[{Cone}(S_p: CF(\phi_1^p,\lambzero) \to CF(\phi_1^p,\lambzero)) \cong CF(\phi_1^p,\lambzero)) \otimes 1 \oplus CF(\phi_1^p,\lambzero) \otimes \theta\] that is identified with the cone differential of $S_p = C_p - R_{1/p},$ \[ d_{Cone}(x_0, x_1) = (d(x_0), -d(x_1) + S_p(x_0)).\] We first construct a canonical $\lambzero$-complex for the cone as follows. As in Section \ref{subsubsec: local FH}, by crossing energy arguments, we may write $d = d_{\loc} + T^{\epsilon_0} {D_d},$ with $D_d$ a $\lambzero$-morphism. Letting $X$ be the part of a $\lamzero$ basis of $CF^{\loc}(\phi_1^p) = CF(\phi_1^p, d_{\loc})$ giving the free part of the homology, the homological perturbation formula gave us a differential $\ol{d}$ on $X,$ and a $\delta_0$-quasi-equivalence obtained from chain maps \[{\ol{\iota}_X: (X,\ol{d}) \to CF(\phi_1^p),} \;\;{\ol{\pi}_X: CF(\phi_1^p) \to (X,\ol{d}),}\] and a homotopy map defined over $\Lambda,$ \[{\ol{\Theta}:CF(\phi_1^p, d_{\loc}) \to CF(\phi_1^p, d_{\loc})}\] given by \eqref{eq: ol Theta} such that $\ol{\iota}_X \circ \ol{\pi}_X = \id + d\ol{\Theta} + \ol{\Theta}d,$ $\ol{\pi}_X \circ \ol{\iota}_X = \id,$ $\ol{\Theta}^2 = 0,$ $\ol{\Theta} \circ \ol{\iota}_X = 0,$ $\ol{\pi}_X \circ \ol{\Theta} = 0.$  Indeed $T^{\delta_0} \ol{\Theta}$ is defined over $\lamzero.$

Consider the complex $(X \oplus X, \ol{d} \oplus -\ol{d}).$ It is $\delta_0$-quasi-equivalent to $(CF(\phi_1^p, d_{\loc}) \oplus CF(\phi_1^p, d_{\loc}), d \oplus -d),$ by $F = \ol{\iota}_X \oplus \ol{\iota}_X,$ $G = \ol{\pi}_X \oplus \ol{\pi}_X,$ $H = \ol{\Theta} \oplus \ol{\Theta}.$

Applying the homological perturbation formula yet again, but now to the splitting \[d_{Cone} = (d \oplus -d) + \til{S},\] where $\til{S} = \til{S}_p$ defined by $\til{S}_p(x_0,x_1) = (0,S_p(x_0)),$ is a square-zero operator $\til{S}^2 = 0,$ we obtain the complex \[C(\phi^p) := Cone(\ol{S}:X \to X) = (X \oplus X, \ol{d}_{Cone}),\] where $\ol{d}_{Cone} = (\ol{d} \oplus -\ol{d})+ \til{\ol{S}}$ for the map $\til{\ol{S}}(\ol{x}_0,\ol{x}_1) = (0, \ol{S}(x_0))$ where $\ol{S}: X \to X$ is given by \[\ol{S} = \ol{\pi}_X S_p \ol{\iota}_X.\] Note that all the sums in the perturbation formula are finite, essentially since $\ol{S}^2 = 0.$ Furthermore, we obtain a $2\delta_0$-quasi equivalence from $\ol{F}: C(\phi^p) \to Cone(S_p),$ $\ol{G}: Cone(S_p) \to C(\phi^p),$ and homotopy $\ol{H}: Cone(S_p) \to Cone(S_p).$ Observe that $\ol{F},\ol{G}$ are not defined over $\lambzero,$ but become so after multiplication by $T^{\delta_0},$ and similarly for $\ol{H}$ and $T^{2\delta_0}.$

Now, extending coefficients to $\lambzero_{\K},$ we write  \begin{equation}\label{eq: Tate splitting into Cone, u loc, u glob} d_{Tate} = d_{Cone} + u E_{\loc} + T^{\epsilon_0} u E,\end{equation} for $\lambzero$-morphisms $E, E_{\loc}$ with all coefficients of all entries of which have positive $u$-degree: that is $E = \Phi(\ol{E})$ with $\ol{E}$ a $\Lambda^0_{\cL}$-morphism (see Proposition \ref{prop: algebraic deformation inequality}), and similarly for $E_{\loc}.$ Furthermore, $E_{\loc}$ is characterized by the local splitting $d_{Tate,\loc} = d_{Cone,\loc} + u E_{\loc}.$ 

Now, it is a quite technical, and rather non-trivial, fact that the homological perturbation formula, applied yet again, this time along $\ol{F},\ol{G}, \ol{H},$ to the splitting \eqref{eq: Tate splitting into Cone, u loc, u glob}, converges and yields a $\Lambda$ but not necessarily a $\lambzero$-morphism. We defer its proof to the end of this section, remarking that the difficulty evidently lies in the term $uE_{\loc}.$ Once convergence is established, we see from the proof of convergence that after multiplication by $T^{\delta_1},$ where $\delta_1 = (4n+2)\delta_0,$  the formula yields a differential $\ol{d}_{Tate}$ that is indeed defined over $\lambzero.$ Similarly, the new injection, projection, and homotopy maps given by the homological perturbation formula, $\ol{\ol{F}}, \ol{\ol{G}}, \ol{\ol{H}}$ converge, and become defined over $\lamzero$ after multiplication by $T^{\delta_1},$  hence yielding a ${2\delta_1}$-quasi-equivalence between $\ol{d}_{Tate}$ and $T^{\delta_1} d_{Tate}.$ The latter is $\delta_1$-quasi-equivalent to $d_{Tate}$ by Section \ref{sec: quasi-eq}.

Finally, we obtain a complex \[\ol{C}_{Tate}(\phi^p) = (C(\phi^p) \otimes \Lambda_{\K}, \ol{d}_{Tate}),\] $3\delta_1$-quasi-equivalent to $(Cone(S_p)\otimes \Lambda_{\K}, d_{Tate}),$  such that $\ol{d}_{Tate}$ has the following form:
\begin{equation}\label{eq: perturbed Tate}
\ol{d}_{Tate} = T^{\delta_1} \ol{d}_{Cone} + u \cl E,
\end{equation}
where $\cl E$ is defined over $\lambzero$ and contains only non-negative powers of $u.$

\begin{proof}[Proof of convergence]
We start by describing relative degrees in local Floer homology. Let $\psi\in \Ham(M,\om)$ have an isolated contractible fixed point $x \in \fix(\psi).$ As remarked above, choosing a lift $\til{\psi} \in \til{\Ham}(M,\om),$ generated by Hamiltonian $H,$ so that $\psi = \phi^1_H,$ and a capping $\ol{x}$ of the periodic orbit $\alpha(x,H)= \{ \phi^t_H(x) \},$ we have well-defined, $\Z$-graded local Floer homology $HF^{\loc}_*(\til{\psi},\ol{x}).$ More precisely, the orbits $\alpha(\til{x},H_1)$ of fixed points $x_1$ of non-degenerate $\psi_1$ sufficiently $C^2$-close to $\psi$ inherit cappings $\ol{x}_1$ from $\ol{x},$ and form a complex as in Section \ref{subsubsec: local FH}. The Conley-Zehnder indices of such $\ol{x}_1$ are within the interval $[\Delta(\til{\psi},\ol{x})-n,\Delta(\til{\psi},\ol{x})+n].$ While the indices themselves depend on the choice of Hamiltonian $H$ and capping $\ol{x},$ their differences do not. Indeed, changing $H,$ or changing the capping, results in a shift in index that is uniform for all $\ol{x}_1.$ 

Hence for each two generators $x_1, y_1 \in CF^{\loc}(\psi_1,x)$ we have a well-defined relative index $\mrm{ind}(x_1,y_1) \in [-2n,2n],$ so that if $dx_1 = y_1,$ then $\mrm{ind}(x_1,y_1) = -1.$ This allows us to define an auxiliary relative degree on $CF^{\loc}(\psi_1,x)$ now considered with $\lambzero$ coefficients, by setting elements of $\lambzero$ to have degree $0.$ 

Observe that for a homogeneous element $x_1 \in CF^{\loc}(\psi_1,x) \otimes \Lambda,$ the homotopy operator $\Theta$ (see Section \ref{subsubsec: local FH}) satisfies \[\mrm{ind}(x_1,\Theta(x_1)) = 1.\]


Now, specializing to $\psi = \phi^p,$ $\psi_1 = \phi_1^p$ as in this section, we observe that \[\ind(x_1,S_{p,\loc}(x_1)) = 0,\] where $S_{p,\loc}$ is the operator defined by $S_p$ in local Floer homology of a given $\zp$-orbit $O,$ while, after extending coefficients and the auxiliary degree to $CF^{\loc}(\phi_1^p) \otimes \Lambda_{\K},$ the homogeneneous components $y_1$ of $uE_{\loc}(x_1)$ satisfy \[\ind(x_1,y_1) \geq 1.\] 

This implies in particular the vanishing of products \begin{equation}\label{eq: index nilpotent} A_1 \cdot \ldots \cdot A_{m} = 0\end{equation} for $m>4n+1,$ and $A_j \in \{\Theta \oplus \Theta, uE_{\loc}, \til{S}_{\loc}\}$ for all $1 \leq j \leq m.$ Here $\til{S}_{\loc} = \til{S}_{p,\loc},$ and we shall further decompose $S_p = S_{p,\loc} + T^{\eps_0} D_S,$ both defined over $\lambzero.$


We claim that the number of terms in the homological perturbation formula defining $\ol{d}_{Tate}$ with valuation at most $k\eps_0$ is finite for all $k \geq 0.$ Indeed, these terms are words in $\{\Theta \oplus \Theta, u E_{\loc}, \til{S}_{\loc}, T^{\eps_0} D_d, T^{\eps_0}\til{D}_S, T^{\eps_0}E\}.$ Letting $F(N)$ be the (finite!) number of terms of length at most $N,$ in view of \eqref{eq: index nilpotent}, the number of terms of valuation at most $k \eps_0$ is bounded by $F((k+1)(4n+2)-1).$ The convergence follows.\end{proof}


\subsubsection{General isolated case}

Here we extend results from Section \ref{subsubsec: local to global non-deg} to the general isolated case. Applying Section \ref{subsubsec: local to global non-deg} to the $\zp$-equivariant product-map \[\cP: C_{Tate}(\Z/(p),CF(\phi_1,\lambzero_{\F_p})^{\otimes p}) \to CF_{Tate,\Z/(p)}(\phi_1^p,\lambzero_{\F_p}),\] we get that $C_{Tate}(\Z/(p),CF(\phi_1,\lambzero_{\F_p})^{\otimes p})$ and $CF_{Tate,\Z/(p)}(\phi_1^p,\lambzero_{\F_p})$ are $\delta_2$-quasi-equivalent, for $\delta_2 = O(\delta_0) \ll \epsilon_0.$ Hence their bar-length spectra are $2\delta_2$-close. Further, by Lemma \ref{lma: quasi Frob}, the bar-length spectrum $\{ \beta_j^{\otimes p}(\phi_1) \}$ of the complex \[C_{Tate}(\Z/(p),CF(\phi_1,\lambzero_{\F_p})^{\otimes p})\] over $\K$ is the doubled $p$-scaled bar-length spectrum of $CF(\phi_1,\lambzero_{\F_p}),$ that is
\begin{align} \label{eq: scaled bar lengths tensor} \beta_{2j}^{\tens p}(\phi_1) &= p \cdot \beta_{j}(\phi_1)\\ \notag
\beta_{2j-1}^{\tens p}(\phi_1) &= p \cdot \beta_{j}(\phi_1)\end{align}
for all $1 \leq j \leq K(\phi_1,\F_p).$ Now, $CF(\phi_1,\lambzero_{\F_p})$ is $\delta_0$-quasi-equivalent, for $\delta_0 \ll \epsilon_0,$ to the canonical $\lamzero$-complex $CF(\phi,\lambzero_{\F_p}).$ Note that again the length of the bar-length spectrum of the latter is bounded from above by $K(\phi)$ independent of $\phi_1.$ Furthermore we obtain that the bar-length spectra of $CF(\phi_1,\lambzero_{\F_p})$  and $CF(\phi,\lambzero_{\F_p})$ are $2\delta_0$-close. Hence, the doubled $p$-scaled bar-length spectrum of $CF(\phi,\lambzero_{\F_p})$ is $p \cdot 2\delta_0$-close to that of $C_{Tate}(\Z/(p),CF(\phi_1,\lambzero_{\F_p})^{\otimes p}).$ Similarly, by the construction of Section \ref{subsubsec: local Tate canonical}, the latter complex $CF_{Tate,\Z/(p)}(\phi_1^p,\lambzero_{\F_p})$ is $3\delta_1$-quasi-equivalent to $\ol{C}_{Tate}(\phi^p).$ 

Combining these observations, we obtain that that the doubled $p$-scaled bar-length spectrum of $CF(\phi,\lambzero_{\F_p})$ is $2\delta_3$-close to the one of $\ol{C}_{Tate}(\phi^p),$ for $\delta_3 = 3\delta_1+ p \cdot \delta_0 + \delta_2.$ Now, we observe that the bar-length spectra of both complexes contain a number of entries bounded by a number $K_T(\phi)$ independent of the perturbation $\phi_1.$ Hence their total bar-lengths satisfy \begin{equation}\label{eq: scaled vs Tate local} 2p \cdot \beta_{\tot}(CF(\phi,\lambzero_{\F_p})) - K_T(\phi) \cdot 2\delta_3 \leq \beta_{\tot}(\ol{C}_{Tate}(\phi^p)).\end{equation} Now, by virtue of Proposition \ref{prop: algebraic deformation inequality} and Lemma \ref{prop: alg cone}, we obtain \[\beta_{\tot}(\ol{C}_{Tate}(\phi^p)) \leq 2 \beta_{\tot}(CF(\phi^p,\lambzero_{\F_p}), T^{\delta_1} d_{\phi^p}).\] Furthermore, the number of entries in the bar-length spectrum of $CF(\phi^p,\lambzero_{\F_p})$ is bounded by $K(\phi^p)$ independent of the perturbation $\phi^p_1,$ hence by Section \ref{sec: quasi-eq} we obtain that \[2 \beta_{\tot}(CF(\phi^p,\lambzero_{\F_p}), T^{\delta_1} d_{\phi^p}) \leq 2 \beta_{\tot}(CF(\phi^p,\lambzero_{\F_p}), d_{\phi^p}) + 4K(\phi^p) \cdot \delta_1.\] Together with \eqref{eq: scaled vs Tate local} this yields that \[ 2p \cdot \beta_{\tot}(CF(\phi,\lambzero_{\F_p})) - K_T(\phi) \cdot 2\delta_3 - 4K(\phi^p) \cdot \delta_1 \leq 2 \beta_{\tot}(CF(\phi^p,\lambzero_{\F_p}), d_{\phi^p}).\] As $\phi_1$ tends to $\phi,$ $\delta_0$ tends to zero, and so does the error term $ K_T(\phi) \cdot 2\delta_3  + 4K(\phi^p) \cdot \delta_1.$ Hence, by the properties of the canonical $\lambzero$-complexes of $\phi$ and $\phi^p$ we therefore obtain \[ 2p \cdot \beta_{\tot}(\phi,\F_p) \leq 2 \beta_{\tot}(\phi^p,\F_p).\] This immediately yields Theorem \ref{thm: beta tot}.  


\begin{rmk}
It is possible to define the Tate bar-length spectrum $\{\wh{\beta}_j(\phi^p)\}$ of $\phi^p$ in the situation above by a limiting procedure, and to prove an analogue of Corollary \ref{cor: qF and SP on beta}, and that the doubled bar-length spectrum of $\phi^p$ majorizes the Tate bar-length spectrum. In particular we would recover the above statement via the estimate \[ 2p\cdot \beta_{\tot}(\phi,\F_p) = \wh{\beta}_{\tot}(\phi^p) \leq 2 \beta_{\tot}(\phi^p,\F_p).\] However, as this is not necessary for our applications herein, we omit the requisite details for this argument. 
\end{rmk}

\begin{rmk}
We expect the methods of this section to apply the question of symplectic mapping class groups of monotone symplectic manifolds, yielding generalizations of \cite[Application 2.15]{Seidel-pants}, \cite[Corollary 1]{SZhao-pants}.
\end{rmk}

\section{Proof of Theorem \ref{thm: main2}}\label{sec: proof}

\begin{proof}[Proof of Theorem \ref{thm: main2}] 
We first assume that $\mrm{char}(\bK) = 0.$ In this case $\bK$ is a field extension of $\Q.$ In view of Section \ref{subsubsec: bar-length three flavors}, the bar-length spectrum of $\phi$ with $\bK$ coefficients coincides with the one with $\Q$ coefficients, and clearly $N(\phi,\bK) = N(\phi,\Q),$ $B(\bK) = B(\Q).$ Similarly, by the Wedderburn theorem (see \cite{EntovPolterovich-semisimple}) if $QH(M,\bK)$ is semisimple, then so is $QH(M,\Q).$ Hence, it is enough to prove the statement for $\Q$ coefficients. 
	
By Lemma \ref{lma:endpoints}, \begin{equation}\label{eq: assumption} N(\phi,\Q) = \sum_{x\in \fix(\phi)} \dim_{\Q} HF^{\loc}(\phi,x) > \dim_{\Q} H_*(M,\Q),\end{equation} implies that $\beta_{\tot}(\phi,\Q) > 0.$ For $p \geq p_0(\phi)$ sufficiently large, $\phi^p$ is an admissible iteration at all the fixed points of $\phi.$ Assume by contradiction that there are no simple fixed points of $\phi^p$ for all $p \geq p_1(\phi)$ sufficiently large. Then by Theorem \ref{thm: GG persistence}, $N(\phi^p,\bK) = N(\phi,\bK)$ for $p \geq \max \{p_0(\phi), p_1(\phi)\}$ for each coefficient field $\bK.$  Furthermore, $N(\phi,\F_p) = N(\phi,\Q),$ $B(\F_p) = B(\Q)$ for $p \geq p_2(\phi),$ by the universal coefficients theorem. In particular for $p \geq \max \{p_0(\phi), p_1(\phi), p_2(\phi)\},$ $N(\phi^p,\F_p) = N(\phi^p,\Q),$ and hence  \[K(\phi^p,\F_p) = \frac{N(\phi^p,\F_p) - B(\F_p)}{2} = \frac{N(\phi^p,\Q) - B(\Q)}{2} = K(\phi^p,\Q).\] Now by Theorem \ref{thm: beta tot}, \[ p \cdot \beta_{\tot}(\phi,\F_p) \leq \beta_{\tot}(\phi^p,\F_p),\] and furthermore, by defintion  \[\beta_{\tot}(\phi^p,\F_p) \leq K(\phi^p,\F_p) \cdot \beta(\phi^p,\F_p).\] Moreover, Theorem \ref{thm: bound semisimple} together with Proposition \ref{prop: bounded beta F_p}, yields $\beta(\phi^p,\F_p) \leq C,$ for $p \geq p_3(\phi)$ and $C = 8n.$  Finally, by Lemma \ref{lemma: beta F_p indep of p}, for $p \geq p_4(\phi),$ $\beta_{\tot}(\phi,\bK) = \beta_{\tot}(\phi,\Q).$ In summary we obtain \[0 < p \cdot \beta_{\tot}(\phi,\Q) \leq K(\phi^p,\Q) \cdot C,\] for $p \geq \max\{p_0(\phi),p_1(\phi),p_2(\phi),p_3(\phi), p_4(\phi)\}.$ Hence \[N(\phi^p,\Q) = 2 \cdot K(\phi^p,\Q) + B(\Q) \geq  p \cdot \frac{2\beta_{\tot}(\phi,\Q)}{C} + B(\Q)\] grows at least linearly in $p$, in contradiction with the identity $N(\phi^p,\Q) = N(\phi,\Q)$ for all $p$ sufficiently large. This proves Theorem \ref{thm: main2} for $\Q$ coefficients.

In a similar fashion to the case of $\mrm{char}(\bK) = 0,$ the case of $\mrm{char}(\bK)= p$ reduces to $\bK = \F_p.$ In this case $N(\phi,\F_p) > B(\F_p)$ implies that $\beta_{\tot}(\phi,\F_p) > 0.$ Hence, fixing $\F_p$ coefficients henceforth, and assuming that $\fix(\phi^{p^k})$ is finite for all $k \geq 1,$ we obtain by Theorem \ref{thm: beta tot} that \[\beta_{\tot}(\phi^{p^k}) \geq p^k \beta_{\tot}(\phi) > 0.\] Now, by Theorem \ref{thm: GG persistence}, there exists $k_0$ such that replacing $\phi$ by $\psi = \phi^{p^{k_0}}$ we have $N(\psi^{p^k},\F_p) = N(\psi,\F_p)$ for all $k \geq 1.$ Furthermore \[\beta_{\tot}(\psi^{p^k},\F_p) \geq p^{k+k_0} \beta_{\tot}(\phi) > 0.\] Hence, as before we obtain \[N(\psi^{p^k}) - B(\F_p) = 2 K(\psi^{p^k}) \geq p^k \cdot C,\] for $C = 2 p^{k_0} \beta_{\tot}(\phi)/ 8n > 0.$ Therefore $N(\psi^{p^k})$ grows at least linearly in $p^k,$ which is a contradiction. 

\end{proof}


\section{Proofs of algebraic results}\label{section: proofs} 

For the convenience of the reader, and for further reference, we recall the proof of the result of Entov-Polterovich \cite{EntovPolterovichCalabiQM}, following the modifications of Ostrover \cite{OstroverAGT}.

\begin{proof}[Proof of Proposition \ref{prop: bounded beta K}]
It is enough to prove that for each $1 \leq j \leq S,$ there is a constant $D(e_j,\bK)$ such that $c(e_j,\til{\phi},\bK) + c(e_j,\til{\phi}^{-1},\bK) \leq D(e_j,\bK).$ Then $D_{E}(\bK) = \max_{1 \leq j\leq S} D(e_j,\bK)$ will be the sought bound. Since by \cite[Section 5.4]{BiranCorneaRigidityUniruling} spectral invariants remain the same after extending coefficients, we shall work with the Novikov field $\Lambda_{M,\tmon}$ with quantum variable $s$ of degree $(-2).$

Applying Poincar\'{e} duality for spectral invariants \cite{EntovPolterovichCalabiQM} we obtain \[c(e_j, \til{\phi}^{-1},\bK) = - \inf \{\,c(b,\til{\phi},\bK)\,|\; \Pi(b,e_j) \neq 0\}. \] 
For each $b$ with $\Pi(b,e_j) \neq 0,$ note that $\Pi(b,e_j) = \Pi(e_j \ast b^{[0]}, [M])$ where $b^{[0]} \in QH_0(M,\bK)$ is the zero-degree component of $b,$ and $c(b,\til{\phi},\bK) \geq c(b^{[0]},\til{\phi},\bK).$ Therefore it is enough to bound $c(b^{[0]},\til{\phi},\bK),$ whenever $\Pi(e_j, b^{[0]}) \neq 0.$ Note now that since $\deg(s) = -2,$ $b^{[0]} s^{-n} \in QH_{2n}(M,\bK),$ \[c(b^{[0]},\til{\phi},\bK) = c(b^{[0]} s^{-n},\til{\phi},\bK) - 2n,\] and $\Pi(e_j,b^{[0]}) \neq 0$ implies $e_j \ast b^{[0]} s^{-n} \neq 0$ in $F_{j,\bK} = e_j \ast QH_{2n}(M,\bK).$ Since $F_{j,\bK}$ is a field, we deduce that $e_j \ast b^{[0]} s^{-n}$ is invertible in $F_{j,\bK}.$ Then its inverse $(e_j \ast b^{[0]} s^{-n})^{-1} \in F_{j,\bK}$ is of the form $e_j \ast c,$ for some $c \in QH_{2n}(M,\bK).$

By the subadditivity property of spectral invariants we obtain \[c(b^{[0]} s^{-n},\til{\phi},\bK) \geq c(e_j \ast b^{[0]} s^{-n},\til{\phi},\bK) - c(e_j,\id,\bK),\] and \[c(e_j \ast b^{[0]}s^{-n},\til{\phi},\bK) \geq c(e_j,\til{\phi},\bK) - c((e_j \ast b^{[0]}s^{-n})^{-1},\id,\bK)\] Combining the last few inequalities, we have \[c(b^{[0]},\til{\phi},\bK) \geq c(e_j,\til{\phi},\bK) - 2n - \cA(e_j) - \cA((e_j \ast b^{[0]}s^{-n})^{-1}),\] whenever $\Pi(b^{[0]},e_j) \neq 0.$ Now it suffices to prove that $\cA(e_j) \leq 2n,$ and $\cA((e_j \ast b^{[0]}s^{-n})^{-1}) \leq 2n.$ This is an immediate corollary of the following statement.

\begin{lma}\label{lma: 2n}
For each element $a \in QH_{2n}(M,\bK)\setminus\{0\},$ $\cA(a) \leq 2n.$
\end{lma}

Deferring the proof of Lemma \ref{lma: 2n}, we summarize that \[c(e_j, \til{\phi}^{-1},\bK) = - \inf \{\,c(b,\til{\phi},\bK)\,|\; \Pi(b,e_j) \neq 0\} \leq \] \[ \leq - \inf \{\,c(b^{[0]},\til{\phi},\bK)\,|\; \Pi(b^{[0]},e_j) \neq 0\} \leq -c(e_j,\til{\phi},\bK) + 6n,\]

 \[c(e_j, \til{\phi}^{-1},\bK) + c(e_j,\til{\phi},\bK) \leq 6n,\] finishing the proof.

\end{proof}

\begin{proof}[Proof of Lemma \ref{lma: 2n}]
This lemma is evident once we notice that \[QH_{2n}(M,\bK) = \bigoplus_{0 \leq k \leq n} H_{2n -2k}(M,\bK) \cdot s^{-k},\] and hence $\cA(a) \leq \cA(s^{-n}) = 2n,$ for all $a \in QH_{2n}(M,\bK) \setminus \{0\}.$
\end{proof}

\begin{proof}[Proof of Proposition \ref{prop: bounded beta F_p}]
Once semisimplicity is shown, the existence of an upper bound $D_{E}(\F_p) \leq 6n$ of $\gamma_E (\phi,\F_p)$ is an application of Proposition \ref{prop: bounded beta K}.

It remains is to show semisimplicity with $(e_{1,p},\dots,e_{S_p,p})$ as in Proposition \ref{prop: bounded beta F_p} and Remark \ref{rmk: reduction mod p of idempotents}.

For semisimplicity we note that $A_j = e_j \ast QH_{2n}(M,\Z) \subset \cO_{F_{j,\Q}}$ is a lattice: it is clearly a finitely generated $\Z$-module, whence it consists of integral elements of $F_{j,\Q}$ (over $\Z$), since $QH_{2n}(M,\Z) \otimes_{\Z} \Q = QH_{2n}(M,\Q),$ we obtain $A_j \otimes_{\Z} \Q = F_{j,\Q}.$  It is a well-known theorem in algebraic number theory that the ring of integers of a number field, such as $\cO_j = \cO_{F_{j,\Q}},$ is a lattice in the number field (here $F_{j,\Q}$). Now consider $\cO_j/A_j.$ This is a torsion abelian group. Hence if $p$ is sufficiently large, then $\cO_j/A_j \otimes_{\Z} \F_p = 0.$ If $p$ is again sufficiently large, so that $\Tor(\cO_j/A_j,\F_p) = 0,$  then the natural map $A_j \otimes_{\Z} \F_p \to \cO_j \otimes_{\Z} \F_p$ induced by the inclusion $A_j \to \cO_j$ is an isomorphism. For $p$ sufficiently large, so that $QH_{2n}(M,\Z)$ has no torsion elements of order $p,$ and the reduction $[e_{j}]_p$ of $e_j$ mod $p$ is well defined, we have $F_{j,\F_p} = [e_{j}]_p \ast QH(M,\F_p) = A_j \otimes_{\Z} \F_p.$ However, for all $p$ that do not divide the discriminant $\Delta_{F_{j,\Q}} \in \Z_{>0}$ of the number field $F_{j,\Q},$ the ring $\cO_{F_{j,\Q}}/p\cO_{F_{j,\Q}} = \cO_{F_{j,\Q}} \otimes_{\Z} \F_p \cong A_j \otimes_{\Z} \F_p = F_{j,\F_p} = [e_{j}]_p \ast QH(M,\F_p)$ is in fact a direct sum of fields (that are hence finite extensions of $\F_p$). Taking $p \geq p_0$ that works for all $e_j$ with $1 \leq j \leq S,$ the unities of the field direct summands of $F_{j,\F_p}$ give idempotents $\{e_{i,p}\}_{i \in I_j},$ such that $(e_{1,p},\ldots,e_{S_p,p})$ split $QH_{2n}(M,\F_p)$ into a direct sum of fields, and satisfy $[e_{j}]_p = \sum_{i \in I_j} e_{i,p}.$

\end{proof}

\begin{proof}[Proof of Lemma \ref{lemma: beta F_p indep of p}]\label{proof: Lemma 1}

We first prove this statement in the case when $\phi$ is non-degenerate. The rational bar-length spectrum of $\phi$ is determined by the multiset of bar-lengths in $V_{r}(\til{\phi})$ with $r \in \{0,\ldots,2N_M-1\}.$ These, in turn, are determined by the $\Q$-linear maps $d_{r,\Q}: CF(\til{\phi},\Q)_{r+1} \to CF(\til{\phi},\Q)_{r}$ of finite dimensional vector spaces over $\Q,$ for $r \in \{0,\ldots, 2N_M -1 \}.$ Furthermore we note that in fact $d_{r,\Q}$ are defined over $\Z$ in the following sense. There exist chain complexes $d_r = d_{r,\Z}: CF(\til{\phi},\Z)_{r+1} \to CF(\til{\phi},\Z)_{r}$ of free $\Z$-modules, such that $CF(\til{\phi},\Q)_{j} = CF(\til{\phi},\Z)_{r} \otimes_{\Z} \Q$ and $d_{r,\Q} = d_r \otimes 1.$

Denote by $B(r) = (e_1,\ldots,e_{N(r,\phi)})$ a $\Z$-basis of $CF(\til{\phi},\Z)_{r}$ given (canonically up to $\{\pm 1\}^{N(j,\phi)}$) by the fixed points of $\phi.$ Then $N(\phi) = \sum_{r = 0}^{2c_M -1} N(r,\phi).$ Furthermore there are orthogonal, in the sense of \cite{UsherZhang}, bases \[P(B(r+1)) =(f_1,\ldots, f_{N(r+1)}),\; Q(B(r)) = (g_1,\ldots, g_{N(r)})\] where $P,Q$ are $\Q$-linear automorphisms of $CF(\til{\phi},\Q)_{r+1},\; CF(\til{\phi},\Q)_{r}$ respectively, that preserve the action filtration, and such that the degree $r$ finite bar-length spectrum is given by $\cA(f_i) - \cA(g_i),$ $d_{r,\Q} f_i = g_i,$ for $1 \leq i \leq K(r,\phi,\Q),$ $d_{r,\Q} f_i = 0,$ for $r \in \{ 0,\ldots, 2N_M -1 \}$ and $i > K(r,\phi,\Q).$ Note that in particular $K(\phi,\Q) = \sum_{r = 0}^{2c_M -1} K(r,\phi,\Q).$

Furthermore, there are large integers $R_1, R_2$ such that $P \in \frac{1}{R_1} End(CF(\til{\phi},\Z)_{r+1}),$ $Q \in \frac{1}{R_1} End(CF(\til{\phi},\Z)_{r}),$ $P^{-1} \in \frac{1}{R_2} End(CF(\til{\phi},\Z)_{r+1}),$ $Q^{-1} \in \frac{1}{R_2} End(CF(\til{\phi},\Z)_{r}).$ This means that there exist maps $P_1 \in End(CF(\til{\phi},\Z)_{r+1}),$ $Q_1 \in  End(CF(\til{\phi},\Z)_{r}),$ $P_{-1} \in  End(CF(\til{\phi},\Z)_{r+1}),$ $Q_{-1} \in  End(CF(\til{\phi},\Z)_{r}),$ such that $P = \frac{1}{R_1} P_1,$ $Q = \frac{1}{R_1} Q_1,$ $P^{-1} = \frac{1}{R_2} P_{-1},$ $Q^{-1} = \frac{1}{R_2} Q_{-1}.$

Choosing a prime $p \gg R_1, R_2$ we can reduce mod $p$ to obtain maps $[P] \in End(CF(\til{\phi},\F_p)_{r+1}),$ $[Q] \in  End(CF(\til{\phi},\F_p)_{r}),$ $[P^{-1}] \in End(CF(\til{\phi},\F_p)_{r+1}),$ $[Q^{-1}] \in  End(CF(\til{\phi},\F_p)_{r}),$ given by $[P] = [R_1]^{-1} [P_1],$ $[Q] = [R_1]^{-1} [Q_1],$ $[P^{-1}] = [R_2]^{-1} [P_{-1}],$ $[Q^{-1}] = [R_2]^{-1} [Q_{-1}],$ where where $[-]$ denotes reduction mod $p.$ Observe that all these maps are action non-increasing. Moreover, $[P] [P^{-1}] = \id,$ $[Q][Q^{-1}] = \id,$ hence these maps are also invertible, and moreover action-preserving. Hence the bases $(f_{i,p}) = [P](B(r+1)),(g_{i,p}) = [Q](B(r))$ are orthogonal, and satisfy the same relations as $(f_i),$ $(g_i),$ but with coefficients in $\F_p.$ Therefore, for such sufficiently large $p,$ the $\F_p$ bar-length spectrum of $\phi$ coincides with the $\Q$ bar-length spectrum of $\phi.$

In the case when $\phi$ is general with isolated fixed points, we consider a Hamiltonian perturbation $\phi_1$ of $\phi$ with Hamiltonian sufficiently $C^2$-small, so that the bars of $\phi_1$ of length at least $\epsilon_0$ have lengths $\epsilon_1$-close to the corresponding bar-lengths of $\phi$ (see Section \ref{subsubsec: barcode isolated}), independently of the coefficient field. By the property of $\epsilon_1,$ this means that the bar-length spectrum of $\phi_1$ over $\epsilon_0$ determines that of $\phi.$ Now, by the non-degenerate case, for $p$ sufficiently large, the bar-length spectrum of $\phi_1$ over $\bF_p$ coincides with that over $\Q.$ This finishes the proof.

\end{proof}

\begin{rmk}
In the above proof we used the Floer complex, keeping the almost complex structure $J$ implicit. While this is not important for our arguments, in fact the lower bound on $p$ for which Lemma 13 holds does not depend on the choice of $J,$ as continuation maps changing $J$ provide filtered graded isomorphisms {\em over $\Z$}. Hence properties of reduction mod $p$ remain invariant.
\end{rmk}

\begin{proof}[Proof of Lemma \ref{lma:endpoints}]\label{proof: Lemma 1}

This is a direct consquence of \cite[Theorem D]{S-Zoll}, building on \cite[Remark 62]{LSV-conj}. Alternatively, passing to the canonical $\lamzero$-complex, and tensoring with $\Lambda,$ this follows from the rank-nullity theorem: \[N = \dim (\ker (d_{\phi})) + \dim (\im (d_{\phi})) = \dim (\ker(d)/ \im(d)) + 2 \dim (\im(d)).\] Hence \[N = B + 2K.\]

\end{proof}


\begin{proof}[Proof of Proposition \ref{prop: beta leq gamma}]
We start with a caveat that simplifies notation. It is beneficial for this proof to be able to work with the (Morse-Bott) degenerate Hamiltonian diffeomorphism $\id_M,$ and the corresponding trivial class $\til{\id}_M$ in the universal cover of $\Ham(M,\om),$ generated by the constant Hamiltonian $H \equiv 0.$ 
This can be done in a few ways, in particular using the approach of cascades of Frauenfelder, or arguing up to $\epsilon,$ and using the PSS isomorphism between the quantum homology of $M$ and the Floer homology of a $C^2$-small time-independent Morse Hamiltonian $h.$

First, following Polterovich-Shelukhin-Stojisavljevic \cite{PolSheSto}, we discuss the action of the quantum homology of $M$ on the Floer persistence module $V_*(\tphi)^t = HF_{*}(\tphi)^{<t}$ for $\tphi \in \til{\Ham}(M,\om).$ There it is shown that each $a \in QH_r(M,\bK),$ with valuation $\nu(a)$ induces a persistence morphism $(a*): V_{m}(\tphi) \to V_{m + r - 2n}(\tphi)[\nu(a)].$ In particular each element $e \in QH_{2n}(M,\bK)$ gives an automorphism (up to shift in valuation) $(e*): V_{m}(\tphi) \to V_{m }(\tphi)[\nu(e)]$ for each $m \in \Z.$ In particular it is evident that for $e,e' \in QH_{2n}(M,\bK),$ $e,e' \neq 0, $ if $\delta(e,e') = \nu(e) + \nu(e') - \nu(e*{e'}) \geq 0$ then, given $e\ast e' \neq 0,$ \[(e\ast) \circ (e'\ast) = sh_{\delta(e,e')} \circ ((e\ast e')\ast),\] and if $e \ast e' = 0,$ then $(e \ast) \circ (e' \ast) = 0$ as a morphism $V_{m}(\tphi) \to V_{m }(\tphi)[\nu(e)+\nu(e')].$ Moreover, a similar statement is true for additive properties: seting $\nu(e,e') = \max\{\nu(e), \nu(e') \},$ we have \[sh_{\nu(e,e') - \nu(e)}\circ(e\ast)+sh_{\nu(e,e') - \nu(e')} \circ (e' \ast) = sh_{\nu(e,e') - \nu(e+e')} \circ ((e+e')\ast),\] if $e+e' \neq 0,$ and $(e+e')\ast = 0$ otherwise.  Finally $[M]\ast = \id_{V_m(\tphi)}.$


Now we let $e_1,\ldots, e_S \in QH_{2n}(M,\bK)$  be such that \[[M] = e_1 + \ldots + e_S\] (recall that $[M] \in QH_{2n}(M,\bK)$ is the unit element), and \[e_{j} \ast e_{k} = e_j \cdot\delta_{j,k}\] for all $1 \leq j \leq k \leq S.$ Here $\delta_{j,k} = 0$ for $j \neq k,$ and $\delta_{j,k} = 1$ for $j=k.$ Set $E=(e_1,\ldots,e_S).$ We claim that there is a natural interleaving between $V_m(\tphi)$ and $\bigoplus_{1 \leq j \leq S} \im (e_j\ast: V_m(\tphi) \to V_m(\tphi)[\nu(e_j)]).$

\begin{lma}\label{lma: split}
The persistence modules $V_m(\tphi)$ and $\bigoplus_{1 \leq j \leq S} \im (e_j\ast: V_m(\tphi) \to V_m(\tphi)[\nu(e_j)]),$ for each $m\in \Z,$ are $(0,\nu(E))$-interleaved.
\end{lma}

\begin{proof}[Proof of Lemma \ref{lma: split}]
Clearly $x \mapsto (e_1\ast x,\ldots, e_S \ast x)$ establishes a morphism \[f_{E}: V_m(\tphi) \to \bigoplus_{1 \leq j \leq S} \im (e_j\ast: V_m(\tphi) \to V_m(\tphi)[\nu(e_j)])\] of persistence modules. Let $\nu(E) = \max_{1\leq j \leq S} \nu(e_j)$ and $\delta(E) = \max_{1 \leq j\leq S} \delta(e_j,e_j).$ Note that since $e_j \ast e_j = e_j$ we have $\delta(e_j,e_j) = \nu(e_j),$ whence $\delta(E) = \nu(E).$ Then for each $1 \leq j \leq S$ there is a map \[sh_{\delta_j} := sh_{\nu(E)-\nu(e_j)}: V_m(\tphi)[\nu(e_j)] \to V_m(\tphi)[\nu(E)].\] Finally we can consider the composition \[g_E: \bigoplus_{1 \leq j \leq S} \im (e_j\ast: V_m(\tphi) \to V_m(\tphi)[\nu(e_j)]) \to V_m(\tphi)[\nu(E)]\] of the inclusion \[\bigoplus_{1 \leq j \leq S} \im (e_j\ast: V_m(\tphi) \to V_m(\tphi)[\nu(e_j)]) \to \bigoplus_{1 \leq j \leq S} V_m(\tphi)[\nu(e_j)]\] the combined shift map \[\bigoplus_{1 \leq j \leq S} V_m(\tphi)[\nu(e_j)] \xrightarrow{(sh_{\delta_1},\ldots, sh_{\delta_S})} \bigoplus_{1 \leq j \leq S} V_m(\tphi)[\nu(E)]\]  and the sum map \[\bigoplus_{1 \leq j \leq S} V_m(\tphi)[\nu(E)] \xrightarrow{\Sigma} V_m(\tphi)[\nu(E)].\]

We claim that $(f_E,g_E)$ is a $(0,\nu(E))$-interleaving.  Indeed, in one direction for each $x \in  V_m(\tphi),$ \[g_E \circ f_E(x) = \sum_{1 \leq j \leq S} sh_{\delta_j} (e_j \ast x) = (\sum_{1 \leq j \leq S} sh_{\delta_j} \circ (e_j \ast)) x =\] \[= sh_{\nu(E)} \circ ([M] \ast) x = sh_{\nu(E)} ([M] \ast x) = sh_{\nu(E)} (x).\]

In the other direction, for each $y =(e_1 \ast y_1,\ldots, e_S \ast y_S) \in \bigoplus_{1 \leq j \leq S} \im (e_j\ast: V_m(\tphi) \to V_m(\tphi)[\nu(e_j)]),$ with $y_j \in V_m(\tphi)$ for all $1 \leq j \leq S,$ 

\[f_E[\nu(E)] \circ g_E (y) = f_E (\sum_{1 \leq j \leq S} sh_{\delta_j}(e_j \ast y_j)) = f_E (\sum_{1 \leq j \leq S} (e_j \ast sh_{\delta_j}(y_j))\] \[= (e_1 \ast (e_1 \ast sh_{\delta_1}(y_1)),\ldots,e_S \ast (e_S \ast sh_{\delta_S}(y_S))) = sh_{\nu(E)}(e_1 \ast y_1,\ldots, e_S \ast y_S).\] The last equality follows from the identity $e_j \ast (e_j \ast (sh_{\delta_j}(y_j))) = sh_{\nu(E)}(e_j \ast y_j)$ for all $1 \leq j \leq S.$ Indeed $e_j \ast (e_j \ast (sh_{\delta_j}(y_j))) = sh_{\delta(e_j,e_j)} \circ (e_j^{\ast 2} \ast) (sh_{\delta_j}(y_j)) = sh_{\nu(e_j)} e_j \ast (sh_{\delta_j}(y_j)) = e_j \ast sh_{\nu(e_j) + \delta_j} (y_j) = e_j \ast sh_{\nu(E)} (y_j) = sh_{\nu(E)}(e_j \ast y_j).$ \end{proof}

As a corollary of Lemma \ref{lma: split}, we obtain by the stability theorem for barcodes that \[\beta(V_m(\tphi)) \leq \max_{1 \leq j \leq S} \beta(\im(e_j \ast: V_m(\tphi) \to V_m(\tphi)[\nu(e_j)])) + \nu(E).\]

Therefore to give a bound on $\beta(V_m(\tphi))$ it is sufficient to bound $\beta(\im(e_j \ast))$ for each $1 \leq j \leq S.$ We shall bound the latter in terms of $\gamma_{e_j}(\tphi).$ 

Consider Hamiltonians $F,G \in \cH$ generating $\phi^1_F = \phi, \; \phi^1_G = \psi,$ $\til{\phi}_F = \til{\phi},\; \til{\phi}_G = \til{\psi}.$ Let $CF(F,J_F),\; CF(G,J_G)$ denote the corresponding Floer complexes. Let $\Delta_{F,G} = G \# \overline{F}$ be the difference Hamiltonian, so that $[\til{\phi}_{\Delta_{F,G} \# F}] = [\til{\phi}_{G}]$ in $\til{\Ham}.$ Fix $\epsilon > 0,$ and arguing up to $\epsilon,$ choose perturbations $F_{\eps},G_{\eps}$ of $F,G$ with $||F_{\eps} - F||_{C^2} \ll \epsilon,\; ||G_{\eps} - G||_{C^2} \ll \epsilon$ such that $F_{\eps},G_{\eps}$ and $\Delta_{F,G,\eps}:= \Delta_{F_{\eps},G_{\eps}}$ are non-degenerate (then $\Delta_{G,F,\eps}:= \Delta_{G_{\eps},F_{\eps}}$ will also be non-degenerate, and $\til{\phi}_{\Delta_{F,G,\eps}} \cdot \til{\phi}_{\Delta_{G,F,\eps}} = 1$ in $\til{\Ham}$), and $||\Delta_{F,G,\eps} - \Delta_{F,G}||_{C^2} \ll \epsilon,$ $||\Delta_{G,F,\eps} - \Delta_{G,F}||_{C^2} \ll \epsilon.$

Let $c_{F,G,\epsilon} \in CF_{2n}(\Delta_{F,G,\eps})$ be an element representing $PSS_{\Delta_{F,G,\eps}}([e_j])$ with \[{\cA_{\Delta_{F,G,\eps}}(c_{F,G,\epsilon}) \leq c(e_j, \Delta_{F,G}) + \eps_0,}\] with $\eps_0 \ll \eps.$ Moreover, reversing the r\^{o}les of $F$ and $G,$ we obtain similarly an element $c_{G,F,\eps} \in CF_{2n}(\Delta_{G,F,\eps})$ representing $PSS_{\Delta_{G,F,\eps}}([e_j])$ with \[\cA_{\Delta_{G,F,\eps}}(c_{G,F,\epsilon}) \leq c(e_j, \Delta_{G,F}) + \eps_0,\] with $\eps_0 \ll \eps.$

Finally we note that $CF(F),CF(F_{\epsilon})$ (and $CF(G),CF(G_{\epsilon})$) are related by interpolation continuation chain maps that shift the action by at most $\eps_0 \ll \eps.$ 

Now we consider the chain map $C_{F,G,\eps,0}: CF(F_{\eps}) \to CF(G_{\eps})$ given by $x \mapsto c_{F,G,\eps} \ast x.$ Composing this with the interpolation-continuation maps we obtain a map $C_{F,G,\eps}: CF(F) \to CF(G)$ satisfying the action estimate $\cA_{G}(C_{F,G,\eps}(x)) \leq c(e_j, \Delta_{F,G}) + \eps_0 + \cA_{F}(x)$ for each $x \in CF(F),$ where $\eps_0 \ll \eps.$ Similarly, we obtain a map $C_{G,F,\eps}: CF(G) \to CF(F)$ satisfying the action estimate $\cA_{F}(C_{G,F,\eps}(y)) \leq c(e_j, \Delta_{G,F}) + \eps_0 + \cA_{G}(y)$ for each $y \in CF(G),$ where $\eps_0 \ll \eps.$

Now we claim that $C_{F,G,\eps},\;C_{G,F,\eps}$ induce a $(\frac{1}{2}\gamma_{e_j}(\til{\phi},\til{\psi}) + \eps)$-interleaving between $I_m(e_j,F) =  \im(e_j\ast:V_m(F) \to V_m(F)[\nu(e_j)])[-\nu(e_j)]$ and a suitable shift of $I_m(e_j,G) = \im(e_j \ast: V_m(G) \to V_m(G)[\nu(e_j)])[-\nu(e_j)].$ While this is not strictly necessary, we added the extra overall shift $[-\nu(e_j)]$ so that $I_m(e_j,F), I_m(e_j,G)$ be submodules of $V_m(F), V_m(G).$ 

We consider $C_{F,G,\eps} \circ C_{G,F,\eps},$ $C_{G,F,\eps} \circ C_{F,G,\eps}.$ Denote by $e_{j,0} \in CF(f,\Lambda_{\bK})$ a chain in the Morse complex computing $QH(M, \bK)$ representing $e_j,$ with $\cA(e_{j,0}) = \nu(e_j).$ We shall often omit the subscript $0,$ and write $e_j$ for $e_{j,0}$ by abuse of notation.

\begin{lma}\label{Lemma: associative}
There exist chain homotopies $R_{F}: CF(F) \to CF(F)[1],$ and $\;R_{G}:CF(G) \to CF(G)[1]$ of degree $1,$ such that:

\[C_{G,F,\eps} \circ C_{F,G,\eps}  = (\ast e_{j,0}) + dR_{F} - R_{F}d,\]

\[C_{F,G,\eps} \circ C_{G,F,\eps}= (\ast e_{j,0}) + dR_{G} - R_{G}d,\]

\[\cA_{F}(R_{F}(x)) \leq c(e_j, \Delta_{G,F},\bK) + c(e_j, \Delta_{F,G},\bK) + \eps_1 + \cA_{F}(x),\]

\[\cA_{G}(R_{G}(y)) \leq c(e_j, \Delta_{G,F},\bK) + c(e_j, \Delta_{F,G},\bK) + \eps_1 + \cA_{G}(y),\]

for each $x \in CF(F),$ $y \in CF(G),$ where $\eps_1 \ll \eps,$ $\eps_1 \geq 2 \eps_0.$

\end{lma}

We note that $c(e_j, \Delta_{G,F},\bK) + c(e_j, \Delta_{F,G},\bK) = \gamma_{e_j,\bK}(\til{\phi}_G \til{\phi}_F^{-1}),$ since $c(e_j, \Delta_{F,G},\bK) = c(e_j, \til{\phi}_G \til{\phi}_F^{-1},\bK)$ and $c(e_j, \Delta_{G,F},\bK) = c(e_j, \til{\phi}_F \til{\phi}_G^{-1},\bK).$ Also it makes sense to remark that $\cA(e_j) = c(e_j,1,\bK) \leq c(e_j, \Delta_{G,F},\bK) + c(e_j, \Delta_{F,G},\bK)$ by the triangle inequality for spectral invariants, as $e_j \ast e_j = e_j.$

Deferring the proof of Lemma \ref{Lemma: associative}, we observe that on the persistence level, for each $m \in \Z,$ the chain maps $C_{F,G,\eps}$ and $C_{G,F,\eps}$ give persistence morphisms \[[C_{F,G,\eps}]: V_m(F) \to V_m(G)[c(e_j, \Delta_{F,G},\bK) + \eps_0],\] \[[C_{G,F,\eps}]: V_m(G) \to V_m(F)[c(e_j, \Delta_{G,F},\bK) + \eps_0],\] that satisfy the following composition relations \[sh_{\eps_1 - 2\eps_0} \circ [C_{G,F,\eps}][c(e_j, \Delta_{F,G},\bK) + \eps_0] \circ [C_{F,G,\eps}] = sh_{\gamma_{e_j,\bK}(\til{\phi}_G \til{\phi}_F^{-1})+\eps_1 - \cA(e_j)} (e_j \ast),\]  \[sh_{\eps_1 - 2\eps_0} \circ [C_{F,G,\eps}][c(e_j, \Delta_{G,F},\bK) + \eps_0] \circ [C_{G,F,\eps}] = sh_{\gamma_{e_j,\bK}(\til{\phi}_G \til{\phi}_F^{-1})+\eps_1 - \cA(e_j)} (e_j \ast),\] where the equalities are of persistence morphisms \[V_m(F) \to V_m(F)[\gamma_{e_j,\bK}(\til{\phi}_G \til{\phi}_F^{-1})+\eps_1],\] \[V_m(G) \to V_m(G)[\gamma_{e_j,\bK}(\til{\phi}_G \til{\phi}_F^{-1})+\eps_1],\] respectively.

It is almost immediate that these maps commute with $e_j\ast:V_m(F) \to V_m(F)[\nu(e_j)],$ and $e_j\ast:V_m(G) \to V_m(G)[\nu(e_j)],$ and hence induce maps \[\Phi'_{F,G,\eps}: I_m(e_j, F) \to I_m(e_j,G)[c(e_j, \Delta_{F,G},\bK) + \eps_0],\] \[\Phi'_{G,F,\eps}: I_m(e_j, G) \to I_m(e_j,F)[c(e_j, \Delta_{G,F},\bK) + \eps_0],\] with the same commutation relations as above, with the additional observation that $(e_j \ast)|_{I_m(e_j,F)} = sh_{\nu(e_j)},$ $(e_j \ast)|_{I_m(e_j,G)} = sh_{\nu(e_j)},$ so that \[sh_{\eps_1 - 2\eps_0} \circ \Phi'_{G,F,\eps}[c(e_j, \Delta_{F,G},\bK) + \eps_0] \circ \Phi'_{F,G,\eps} = sh_{\gamma_{e_j,\bK}(\til{\phi}_G \til{\phi}_F^{-1})+\eps_1},\]  \[sh_{\eps_1 - 2\eps_0} \circ \Phi'_{F,G,\eps}[c(e_j, \Delta_{G,F},\bK) + \eps_0] \circ \Phi'_{G,F,\eps} = sh_{\gamma_{e_j,\bK}(\til{\phi}_G \til{\phi}_F^{-1})+\eps_1},\] where the equalities are of persistence morphisms \[I_m(e_j,F) \to I_m(e_j,F)[\gamma_{e_j,\bK}(\til{\phi}_G \til{\phi}_F^{-1})+\eps_1],\] \[I_m(e_j, G) \to I_m(e_j, G)[\gamma_{e_j,\bK}(\til{\phi}_G \til{\phi}_F^{-1})+\eps_1],\] respectively.

Define $\overline{\gamma}_{e_j,\bK}(\til{\phi}_G \til{\phi}_F^{-1}) = \frac{1}{2}(-c(e_j,\til{\phi}_G \til{\phi}_F^{-1},\bK) + c(e_j,\til{\phi}_F \til{\phi}_G^{-1},\bK)).$

Post-composing each of the maps $\Phi'_{F,G,\eps},\; \Phi'_{G,F,\eps}$ with $sh_{\frac{1}{2} \eps_1 - \eps_0},$ and restating them in terms of the persistence modules \[\overline{I}_m(e_j,F) = I_m(e_j,F),\] \[\overline{I}_m(e_j,G) = I_m(e_j,G)[\frac{1}{2}\overline{\gamma}_{e_j,\bK}(\til{\phi}_G \til{\phi}_F^{-1})]\] we obtain maps \[\Phi_{F,G,\eps}: \overline{I}_m(e_j, F) \to \overline{I}_m(e_j,G)[\frac{1}{2}\gamma_{e_j,\bK}(\til{\phi}_G \til{\phi}_F^{-1}) + \frac{1}{2}\eps_1],\] \[\Phi_{G,F,\eps}: \overline{I}_m(e_j, G) \to \overline{I}_m(e_j,F)[\frac{1}{2}\gamma_{e_j,\bK}(\til{\phi}_G \til{\phi}_F^{-1}) + \frac{1}{2}\eps_1],\] with  \[ \Phi_{G,F,\eps}[\frac{1}{2}\gamma_{e_j,\bK}(\til{\phi}_G \til{\phi}_F^{-1}) + \frac{1}{2}\eps_1] \circ \Phi_{F,G,\eps} = sh_{\gamma_{e_j,\bK}(\til{\phi}_G \til{\phi}_F^{-1})+\eps_1},\]  \[\Phi_{F,G,\eps}[\frac{1}{2}\gamma_{e_j,\bK}(\til{\phi}_G \til{\phi}_F^{-1}) + \frac{1}{2}\eps_1] \circ \Phi_{G,F,\eps} = sh_{\gamma_{e_j,\bK}(\til{\phi}_G \til{\phi}_F^{-1})+\eps_1},\]

This means that $\Phi_{F,G,\eps},$ $\Phi_{G,F,\eps}$ give a $(\frac{1}{2}\gamma_{e_j,\bK}(\til{\phi}_G \til{\phi}_F^{-1}) + \frac{1}{2}\eps_1)$-interleaving between $\overline{I}_m(e_j,F),$ and $\overline{I}_m(e_j,G).$ However, since $\eps_1 \ll \eps,$ and $\eps$ was arbitrarily chosen, we see that the interleaving distance between $\overline{I}_m(e_j,F),$ and $\overline{I}_m(e_j,G)$ is $\frac{1}{2}\gamma_{e_j,\bK}(\til{\phi}_G \til{\phi}_F^{-1}).$ The statements on the bar-length spectrum, and the bottle-neck distance on barcodes modulo $\R$-shifts now follow immediately.

\end{proof}

\begin{proof}[Proof of Lemma \ref{Lemma: associative}]

The proof consists of a few steps, all of which essentially amount to proving the associativity of the pair-of-pants product in Floer homology, taking care of the action-filtration levels. We refer to \cite{OhBook, SeidelBook}, and references therein, for the technical framework related to associativity, and focus on the filtration aspects below. Related arguments appear in \cite{KS-bounds} in the case of Lagrangian Floer homology.

First we claim that for each two cycles $a,b \in CF(\Delta_{F,G,\eps}), CF(\Delta_{G,F,\eps})$ there is a well-defined product $a \ast b \in QH(M;\bK)$ and there exist chain-homotopies $R'_{a,b},R'_{b,a}$ such that 

\[b \ast (a \ast -) = (b\ast a) \ast (-) + dR'_{F,b,a} - R'_{F,b,a}d,\]

\[a \ast (b \ast -) = (a\ast b) \ast (-) + dR'_{G,a,b} - R'_{G,a,b}d,\]

\[\cA_{F}(R'_{F,a,b}(x)) \leq \cA_{\Delta_{F,G,\eps}}(a) + \cA_{\Delta_{F,G,\eps}}(b) + \eps_0 + \cA_{F}(x),\]

\[\cA_{G}(R'_{G,a,b}(y)) \leq \cA_{\Delta_{F,G,\eps}}(a) + \cA_{\Delta_{F,G,\eps}}(b) + \eps_0 + \cA_{G}(y),\]

\[\eps_0 \ll \eps.\]

Since the constructions of the homotopies $R'_{F,b,a}$ and $R'_{G,a,b}$ are analogous, we only consider the first one. We construct $R'_{F,b,a}$ by considering a parametric Floer equation over a one-parameter family of $4$-punctured spheres over an open interval, which we identify with $\cR = \R,$ with cylindrical ends near the punctures and conformal structures degenerating in two different ways. We arrange the punctures to lie on one circle in the sphere, and the cylindrical ends to have asymptotic markers along this circle. This gives a homotopy-canonical choice of asymptotic ends along the family. There are $3$ input ends $\psi^{-}_{\De_{G,F,\eps}}, \psi^{-}_{\De_{F,G,\eps}}, \psi^{-}_{F}: Z_{-} \to S^2$, and one output end $\psi^{+}_{F}:Z_{+} \to S^2$, where $Z_{-} = (-\infty,0] \times S^1,$ $Z_{+} = [0,\infty) \times S^1,$ endowed with the standard complex structure. All these maps are which are proper holomorphic embeddings. We choose Hamiltonian perturbation data $F_{\rho}$ for $\rho \in \cR,$ that agrees with $X_H \otimes dt$ on each end $\psi_{\pm,H}$ (this shall be our convention in the whole discussion). Moreover we require $F_{\rho}$ to depend on $\rho \in \cR$ in the following way. For $\rho \ll -1$ it has trivial curvature, for $\rho \gg 1$ it has curvature at most $\eps_0,$ and in these two cases it agrees with the following Hamiltonian perturbations obtained by gluind two $3$-punctured spheres, equipped with Floer data. For $\rho \ll -1,$ the first sphere has two inputs ends $\psi^{-}_{\De_{F,G,\eps}},\psi^{-}_{F}:Z_{-} \to S^2$ and one output end $\psi^{+}_{\Delta_{F,G,\eps} \# F}:Z_{+} \to \R,$ and the second sphere has two inputs ends $\psi^{-}_{\Delta_{G,F,\eps}},\psi^{-}_{\Delta_{F,G,\eps} \# F}:Z_{-} \to S^2$ and one output end $\psi^{+}_{F} :Z_{+} \to \R.$ We glue these two spheres along the ends $\psi^{\pm}_{\Delta_{F,G,\eps}\# F} ,$ with gluing length $l = |\rho|,$ and note that since the Hamiltonian perturbations agree over these ends, we get a Hamiltonian perturbation on the glued $4$-punctured sphere with cylindrical ends. For $\rho \gg 1$ the situation is identical, except that the two spheres with cylindrical ends are as follows. Fix a small Morse function $h$ of $C^0$-norm at most $\eps'_0 \ll \eps_0.$ The first sphere has two inputs ends $\psi^{-}_{\Delta_{G,F,\eps}},\psi^{-}_{\Delta_{F,G,\eps}}:Z_{-} \to S^2$ and one output end $\psi^{+}_{h} :Z_{+} \to \R,$ and the second sphere has two inputs ends $\psi^{-}_{h},\psi^{-}_{F}:Z_{-} \to S^2$ and one output end $\psi^{+}_{F} :Z_{+} \to \R.$ We glue these two spheres along the ends $\psi^{\pm}_{h} ,$ over which the Hamiltonian perturbation agree, with gluing length $l = |\rho|.$ It is easy to see that such Floer data with the prescribed curvature assumptions exist. A generic choice $I$ of domain (and parameter) dependent almost complex structures makes the index $0$ and $1$ solution space of these parametric equations transversely cut out. We require $I$ to agree with the $l = |\rho|$-gluing of regular domain dependent almost complex structures on the relevant $3$-punctured spheres at the ends $\rho \ll -1, \rho \gg 1.$  Finally, for $l$ sufficiently large we see that the index $0$ solutions for the equaltion on $S_{\rho}$ give compositions of operations corresponding to the $3$-punctured spheres. Plugging $b \otimes a$ into the two corresponding inputs, we get $b \ast (a \ast -)$ for $\rho \ll -1,$ and $(b \ast' a) \ast'' (-).$ Here the operations are $\ast': CF(\De_{G,F,\eps}) \otimes CF(\De_{F,G,\eps}) \to CF(h)$ and $\ast'': CF(h) \otimes CF(F) \to CF(F),$ given by the relevant $3$-punctured spheres. Considering compactifications of index $1$ solution space to the parametric equation we obtain a chain homotopy $R'_{F,b,a}:CF(F) \to CF(F)[1]$ satisfying \[b \ast (a \ast -) = (b\ast' a) \ast'' (-) + dR'_{F,b,a} - R'_{F,b,a}d,\]
\[\cA_{F}(R'_{F,a,b}(x)) \leq \cA_{\Delta_{F,G,\eps}}(a) + \cA_{\Delta_{F,G,\eps}}(b) + \eps'_0 + \cA_{F}(x),\] \[\eps'_0 \ll \eps_0 \ll \eps.\]

Finally, we consider the PSS isomophisms $\Psi_{PSS}: CM(f,\Lambda_{\bK}) \to CF(h)$ and $\overline{\Psi}_{PSS}: CF(h) \to CM(f,\Lambda_{\bK}).$  It is a standard action estimate to show that \[\cA_h(\Psi_{PSS}(x)) \leq \cA(x) + \eps'_0\] \[\cA_h(\overline{\Psi}_{PSS}(y)) \leq \cA(y) + \eps'_0\]

for each $x \in CM(f,\Lambda_{\bK}),$ and $y \in CF(h),$ and moreover there exists a chain homotopy $R^{PSS}_h$ such that 

\[\Psi_{PSS} \circ \overline{\Psi}_{PSS} = 1 + d R^{PSS}_h - R^{PSS}_h d,\]

\[\cA_{h}(R^{PSS}_h(y)) \leq \cA_{h}(y) + 2 \eps_0',\]

for each $y \in CF(h).$

Finally, by gluing and homotopy of domain-dependend almost complex structures, we see that $\overline{\Psi}_{PSS}(a \ast' b)$ and  ${\Psi}_{PSS}(x) \otimes'' (-)$ are homotopic, respectively, to $a \ast b$ and $x \ast (-),$ with homotopies ($r_{a,b}$ and $R_x$) that do not increase the action by more that $\eps_0'$ (we remark that if one wishes, one can define $a \ast b$ and $x \ast (-)$ in this way).

We conclude that $R'_{F,b,a} = R''_{F,b,a} + (R^{PSS}_h(a \ast' b)) \ast''(-) + R_{\overline{\Psi}(a \ast' b)} + (r_{a,b}) \ast (-)$ gives the required homotopy, satisfying the requried action estimates.


Now we proceed by noting first that by the compatibility of the pair-of-pants product and the quantum product under the PSS isomorphism, \[[c_{G,F,\eps} \ast c_{F,G,\eps}] = e_j \ast e_j = e_j = [e_{j,0}].\] Hence \[g = c_{G,F,\eps} \ast c_{F,G,\eps} - e_{j,0}\] is a boundary in the Morse complex $CM(f,\Lambda_{\bK}).$ Moreover \[\cA(g) \leq \max\{\gamma_{e_j}(\til{\phi}_G \til{\phi}_F^{-1})+2\eps_0, \nu(e_j)\} = \gamma_{e_j}(\til{\phi}_G \til{\phi}_F^{-1})+2\eps_0,\] since $\gamma_{e_j}(\til{\phi}_G \til{\phi}_F^{-1}) \geq \nu(e_j).$ However, since $\beta(CM(f,\Lambda_{\bK})) = 0,$ there exists an element $r_F \in CM(f,\Lambda_{\bK})$ with \[\cA(r_F) \leq \gamma_{e_j}(\til{\phi}_G \til{\phi}_F^{-1})+2\eps_0\] and $g = c_{G,F,\eps} \ast c_{F,G,\eps} - e_{j,0} = d(r_F).$

Finally, by the Leibnitz rule \[(r_F \ast): CF(F) \to CF(F)[1] \] gives a homotopy between $(c_{G,F,\eps} \ast c_{F,G,\eps}) \ast (-)$ and $e_{j,0} \ast (-).$ Therefore \[R_{F} = R'_{c_{G,F,\eps},c_{F,G,\eps}} + (r_F \ast)\] gives a homotopy between $c_{G,F,\eps} \ast (c_{F,G,\eps} \ast (-))$ and $e_{j,0} \ast (-)$ and satisfies the action estimate \[\cA_{F}(R_{F}(x)) \leq \gamma_{e_j}(\til{\phi}_G \til{\phi}_F^{-1})+2\eps_0 + \cA_{F}(x),\] as required. The construction of $R_{G}$ is identical. The proof of the lemma is now finished.
\end{proof}

\begin{proof}[Proof of Proposition \ref{prop: algebraic deformation inequality}]
Recall that we work over $\Lambda = \Lambda_{\bK}$ and its extension $
\Lambda'=\Lambda_{\cK},$ where $\cK = \bK((u)):= \bK[u^{-1},u]].$
Now consider a Floer-type complex $(C,d)$ over $(\Lambda')^0$ that extends a Floer-type complex $(C_0,d_0)$ over $\Lambda,$ so that $C = {C_0 \otimes_{\Lambda} \Lambda'}$ and $d = d_0 + u d_1 + \ldots,$ for certain $\lambzero$-linear maps $d_j: C_0 \to C_0,$ $j \geq 1$ (that, other than $d_1,$ are not necessarily chain maps). Furthermore, we know that $\dim_{\Lambda'} H(C,d) = \dim_{\Lambda} H(C_0,d_0).$ We wish to compare the torsion exponents of $(C,d)$ to those of $(C_0,d_0).$

We recall that torsion exponents can be obtained from the coefficients in the Smith normal forms of the differentials considered as matrices in $Mat(N,\lambzero),$ $Mat(N,(\Lambda')^0)$ respectively.

It is a classical theorem in the theory of principal ideal domains that these coefficients, in the form of invariant divisors, $(c_1,\ldots,c_N),\; (c'_1,\ldots,c'_N)$ arranged in increasing order, so that $c_i$ (resp. $c'_i$) divides $c_j$ (resp. $c'_j$) for $i \leq j$, (so that the actual coefficients $(b_1,\ldots,b_N),\; (b'_1,\ldots,b'_N),$ the elementary factors, are given by the formulae $c_j = b_1 \cdot \ldots \cdot b_j,$ for all $1 \leq j \leq N$) are given by considering the minors of the corresponding matrices. More precisely $c_j$ is given as a (principal) generator of the ideal generated by all the determinants of $j\times j$ minors of the matrix. While $\lambzero$ is not a principal ideal domain (indeed $\Lambda_{+}$ is an ideal that is not principal), it has the property that each {\em finitely generated} ideal in it is principal. It is easy to see that the above description of the invariant divisors holds for such rings.

Finally we note that in $\lambzero$ each principal ideal $I$ is given as $I = T^{\gamma} \lambzero$ for some $\gamma \geq 0.$ We can recover $\gamma$ as $\gamma = \min \{ val(\lambda)|\; \lambda \in I \}.$ Recall that $val(0) = +\infty.$

This implies immediately that for all $1 \leq j \leq N$ the corresponding invariant divisor is $T^{\gamma_j},$ with $\gamma_j = \beta_1 + \ldots + \beta_j,$ satisfying\[\gamma_j = \min \{val(\det M)\,|\;M\;\text{is a}\; j\text{-minor of}\;d_0\}.\]

Similarly $\gamma'_j = \beta'_1 + \ldots + \beta'_j,$ satisfies \[\gamma'_j = \min \{val(\det M')\,|\;M'\;\text{is a}\; j\text{-minor of}\;d\}.\] Here we use the convention that the minimum of an empty set of non-negative real numbers is $+\infty.$ 

Now we note that since $d = d_0 + uD,$ with all the matrix entries of $D$ lying in $\Lambda^0_{\cl{L}},$ for each $j$-minor $M'$ of $d,$ we have the relation: \[\det M' = \det M + u \lambda,\] for the corresponding $j$-minor $M$ of $d_0$ and a certain coefficient $\lambda \in \Lambda^0_{\cl{L}}.$  Observing that $\det M \in \lambzero,$ while all coefficients of $u\lambda$ have positivie $u$-degree, this immediately implies that \[val(\det M') \leq val(\det M).\] Therefore \[\gamma'_j \leq \gamma_j\] for all $1 \leq j \leq N,$ whence by the definition of $\{\gamma_{j},\,\gamma'_{j}\}_{1 \leq {j} \leq N}$ we obtain \[\beta'_1 + \ldots \beta'_j \leq \beta_1 + \ldots +\beta_j \] for each $1 \leq j\leq N.$ This finishes the proof.
\end{proof}

\begin{proof}[Proof of Lemma \ref{prop: alg cone}]
The proof is similar to that of Proposition \ref{prop: algebraic deformation inequality}. Indeed, choose for concreteness the orthonormal basis $E = \{\xi_i, \zeta_j, \eta_j\}$ for $1 \leq i \leq B=B(C),$ $1 \leq j \leq K=K(C)$ with $d(\zeta_j) = T^{\beta_j} \eta_j$ for $1 \leq j \leq K,$ where $\beta_j = \beta_j(C).$ Set $\beta_{\tot} = \beta_{\tot}(C) = \beta_1 + \ldots + \beta_K,$ and $N = N(C) = \dim_{\Lambda} C \otimes_{\Lambda}.$ As a $\lambzero$-module $Cone(S) = C \oplus C.$ The differential in $Cone(S)$ is given by $d_{Cone}:(x_0,x_1) \mapsto (d(x_0) ,-d(x_1) + S(x_0)).$ Note that since \[N' = N(Cone(S)) = \dim_{\Lambda} Cone(S) \otimes \Lambda = 2 N,\] and since $S$ induces the zero map on homology, \[B'=B(Cone(S)) = \dim_{\Lambda} H(Cone(S)) \otimes \Lambda = 2 \dim_{\Lambda} H(C) \otimes \Lambda = 2B.\] Hence, there are $K' = (N'-B')/2 = N-B = 2K$ verbose bar-lengths \[0 \leq \beta'_1 \leq \ldots \leq \beta'_{2K}\] in $Cone(S).$

The matrix of $d_{Cone}$ in the basis $\ol{E}$ of $C \oplus C$ induced by $E$ is of two-by-two block form, with the upper off-diagonal block being zero, and the diagonal blocks being the matrices $[d]_{E}$ and $[-d]_{E}$ respectively. Let $M$ be a $K$-minor of $[d]_E$ with $val(\det M) = \beta_{\tot}.$ Then the $2K$-minor $\ol{M}$ of $[d_{Cone}]_{\ol{E}}$ corresponding to the rows and columns of $M$ taken twice, is again of two-by-two block form, with $M$ and $-M$ being the diagonal blocks, the upper off-diagonal block being zero. Hence \[val \det(\ol{M}) = 2 \cdot val(\det M) = 2 \cdot \beta_{\tot}(C),\] whence as in the proof of Proposition \ref{prop: algebraic deformation inequality} we obtain that \[\beta_{\tot}(Cone(S)) = \beta'_1 + \ldots + \beta'_{2K} \leq 2 \cdot \beta_{\tot}(C).\]
\end{proof}

\section*{Acknowledgements}
I gratefully acknowledge my teacher Leonid Polterovich for his continuous support and, along with Paul Biran, Octav Cornea, Asaf Kislev, Vuka\v{s}in Stojisavljevi\'{c}, and Jingyu Zhao, for fruitful collaborations during which I learnt most of the tools that I use in this paper.

I thank Viktor Ginzburg, Ba\c{s}ak G\"{u}rel, Paul Seidel, and Nicholas Vichery, for very useful conversations at different stages of this project, and Felix Schlenk for comments on the exposition.

This work was initiated and partially carried out during my stay as a member in the Institute for Advanced Study, where I was supported by NSF grant No. DMS-1128155. I thank the IAS for a great research atmosphere, and Helmut Hofer for his hospitality, and interest in this project.

At the University of Montr\'{e}al I was supported by an NSERC Discovery Grant and by the Fonds de recherche du Qu\'{e}bec - Nature et technologies.


\bibliographystyle{abbrv}
\bibliography{bibliographyHZ}

\end{document}